\documentclass[12pt,english,russian]{article}
\usepackage[T2A]{fontenc}
\usepackage[utf8]{inputenc}
\usepackage[dvipsnames]{xcolor}
\usepackage{amssymb,amsfonts,amsmath,mathtext,amsthm}
\usepackage{cite,enumerate,float,indentfirst}
\usepackage{graphics,epsfig,graphicx,epstopdf,amscd,latexsym,verbatim,hyperref,ulem}
\usepackage[top = 30mm, bottom = 20mm, left = 30mm, right = 25mm]{geometry}
\usepackage{centernot}

\usepackage{pgfplots,comment}
\pgfplotsset{compat=1.9}
\usepackage{filecontents}
\usetikzlibrary {arrows.meta}
\usepackage{blindtext, rotating}

\newtheorem{lemma}{Lemma}[section]
\newtheorem{example}{Example}[section]
\newtheorem{theorem}{Theorem}[section]
\newtheorem{corollary}{Corollary}[section]
\newtheorem{proposition}{Proposition}[section]
\newtheorem{definition}{Definition}[section]
\newtheorem{remark}{Remark}[section]
\newtheorem{problem}{Problem}[section]

\def\Imm{\mathrm{Im}\,}

\begin{document}
\begin{center}
{\Large Linear-in-degree monomial Rota---Baxter of weight zero and averaging operators on $F[x, y]$ and $F_0[x, y]$}

\smallskip

Artem Khodzitskii 
\end{center}

\begin{abstract}
Rota---Baxter operators on the polynomial algebra have been actively studied
since the work of S.H.~Zheng, L.~Guo, and M.~Rosenkranz (2015).
Monomial operators of an arbitrary weight (2016),
as well as injective operators of weight zero on $F[x]$ (2021), have been described.
The author described monomial Rota---Baxter operators of nonzero weight on $F[x, y]$ coming from averaging operators (2023)
and studied the connection between monomial Rota---Baxter operators and averaging operators (2024). The main result of the current work is the classification of monomial Rota---Baxter operators of weight zero on $F[x,y]$ 
coming from monomial linear-in-degree averaging operators.

\medskip
{\it Keywords}:
Rota---Baxter operator, averaging operator, polynomial algebra.

MSC code: 16W99
\end{abstract}

\section{Introduction}

Rota---Baxter operators constitute an algebraic abstraction of the integral operator 
and have proven to be fundamental in various areas of mathematics and theoretical physics. 
These operators play a significant role in the study of different versions of the Yang---Baxter equation, 
combinatorics, and quantum field theory~\cite{GuoMonograph}. 
Originally introduced by G.~Baxter in 1960~\cite{Baxter},
the relations similar to the one defining Rota---Baxter operators
can be found in the earlier works by F.\,G.~Tricomi~\cite{Tricomi} and M.~Cotlar~\cite{Cotlar}.

\begin{definition}
A linear operator $R$ on an algebra $A$ defined over a field $F$ 
is called a~Rota---Baxter operator (an RB-operator, for short)
if the following relation
\begin{gather}\label{RBO}
R(a)R(b) = R\big(R(a)b + a R(b) + \lambda a b\big)
\end{gather}
holds for all $a, b \in A$.
Here $\lambda\in F$ is a fixed scalar called a weight of $R$.
\end{definition}

When $\lambda = 0$, relation~\eqref{RBO} is a generalization of 
the integration by parts formula.
G.-C. Rota played an important role both in the study and in the promotion of such operators during the period from the 1960s to the 1990s~\cite{Rota}.
In the 1980s, these operators were independently rediscovered in the context of 
the classical and modified Yang---Baxter equations arising in mathematical physics~\cite{BelaDrin82,Semenov83}. 
At present, Rota---Baxter operators are the subject of active research in a wide range of algebraic settings,
including groups and other algebraic systems.

In this work, we are primarily interested in the relationship between Rota---Baxter operators and averaging operators. 
Averaging operators were originally introduced by O.~Reynolds in the context of turbulence theory~\cite{Reynolds}. 
Subsequently, J.~Kamp\'{e} de F\'{e}riet recognized their importance and conducted an extensive study 
of averaging and Reynolds operators over the course of three decades~\cite{KampeDeFeriet}.

\begin{definition}
A linear operator $T$ on an algebra $A$ defined over a field $F$
is called an averaging operator if the following relations hold for all $a,b \in A$:
\begin{gather}\label{Averaging}
T(a)T(b) = T(T(a)b) = T(a T(b)).
\end{gather}
\end{definition}

Numerous works have been devoted to averaging operators,
and a more comprehensive historical overview can be found in~\cite{PeiGuoAveraging}. 
Below, we highlight several relevant contributions. 
In~\cite{RotaAvOp}, averaging operators were studied in the context of functional analysis, 
while their relationship to conditional expectations was established in~\cite{Moy}. 
The connection between averaging operators and Rota---Baxter operators was noticed by N.~Bong in~\cite{Bong}.

In this work, we focus on Rota---Baxter operators of weight zero 
that are coming from linear-in-degree monomial averaging operators on the algebras $F[x, y]$ and $F_0[x, y]$. 
Here, $F_0[x, y]$ denotes the ideal in $F[x,y]$ generated by the set $\{x, y\}$.

The active study of Rota---Baxter operators on polynomial algebras began 
with the work of S.\,H.~Zheng, L.~Guo, and M.~Rosenkranz in 2015~\cite{Monom2}. 
In that work, injective monomial Rota---Baxter operators of weight zero on $F[x]$ were described. 
The monomiality condition means that each monomial is mapped to another monomial multiplied by a~scalar. 
Monomial Rota---Baxter operators of an arbitrary weight $\lambda$ were described on $F[x]$ in~\cite{Monom} 
and on the non-unital algebra $F_0[x]$ in~\cite{MonomNonunital}. 
Further progress was achieved in~\cite{GubPer}, where the conjecture of 
S.\,H. Zheng, L.~Guo, and M.~Rosenkranz concerning injective (not necessarily monomial) 
Rota---Baxter operators on $F[x]$, where $F$ is a field of characteristic zero, was confirmed.
Note that recent studies of integro-differential rings~\cite{RaabRegensburger, DuRaab, RaabRegensburger2} 
involve polynomial algebras as well as Laurent polynomial algebras.

In~\cite{ReprRB,ReprRB2,ReprRB3}, finite-dimensional irreducible representations of 
certain Rota---Baxter algebras defined on $F[x]$ were described. 
Interesting examples of Rota---Baxter operators on $F[x, y]$ appeared in~\cite{Ogievetsky,Viellard-Baron}. 
In~\cite{Khodzitskii}, the author described monomial Rota---Baxter 
operators of nonzero weight on $F[x, y]$ coming from averaging operators. 
In~\cite{Khodzitskii2}, various classes of monomial Rota---Baxter 
and averaging operators on $F[x,y]$ and $F_0[x,y]$ were studied. 
In the last work, methods for constructing a monomial Rota---Baxter operator of weight zero 
by a monomial averaging operator and vice versa were presented.

It is natural to study monomial linear operators on the polynomial algebra $F[X]$, 
since monomial derivations on it have been investigated for several decades~\cite{NowZiel,Ollagnier,Kitazawa,Essen}. 
In particular, monomial derivations were considered in the context of the Hilbert's 14th problem~\cite{Essen}. 
It is worth noting that if a Rota---Baxter operator is invertible, then its inverse is a derivation.

In the present work, we classify Rota---Baxter operators coming from 
monomial linear-in-degree averaging operators on $F[x, y]$ and $F_0[x, y]$. 
A monomial averaging operator~$T$ is called linear-in-degree, if 
$\deg_x T(x^ny^m)$ and $\deg_y T(x^ny^m)$ linearly depend on $n$ and $m$.

The work is organized as follows. 

In \S2, we provide required preliminaries.
The main part of the latter is devoted to the recurrent relations of a very certain type, and the proofs of the statements on them are given in Appendix.

In \S3, we consider the problem (Problem~\ref{MainProblem} below)
about the classification of monomial RB-operators~$R$ of weight zero 
on $F[x,y]$ and $F_0[x,y]$ having the form 
$$
R(z) = \alpha_z T(z),\ z \in M(X)\,(M_0(X)),\ \alpha_z \in F, 
$$
where $T$ is a fixed monomial linear-in-degree averaging operator.
This formulation allows us to reformulate the problem 
in terms of solving a (non-linear) recurrence relation on the coefficients~$\alpha_z$.
Up to multiplication by a nonzero scalar from $F$,
all monomial linear-in-degree averaging operators on $F[x, y]$ and $F_0[x, y]$ are described in \S3.
 
The paragraphs \S4 and \S5 are devoted to the solution
of Problem~\ref{MainProblem}. 
In \S4, we describe RB-operators of weight zero
of the form $R(x^n y^m) = \alpha_{n, m} y^{m + rn + c}$
and $R(x^n y^m) = \alpha_{n, m} x^{r(m + c)} y^{m + c}$, 
where $r \geqslant 0$, $c \geqslant \nu(0)$.

In \S5, we classify RB-operators of weight zero of the form 
$R(x^n y^m) = \alpha_{n, m} x^{n + p_x} y^{m + p_y}$, where $p_x, p_y \geqslant 0$. 
We consider two subcases: when $p_x = p_y = 0$
and $p_x + p_y \geqslant 0$.

In Appendix, we provide all proofs of the results devoted to recurrence relations (Appendix A) and the proof of Corollary~\ref{RBO_Case3_corollary} (Appendix B).

\section{Preliminaries}

We denote by $X$ a set of variables, 
where $X = \{ x_i \mid i \in I \}$ and $I \neq \emptyset$ is an index set. 
The notation $F_0[X] = (X)$ refers to the ideal of $F[X]$ generated by~$X$, 
and we write $F^* = F \setminus \{0\}$ for the set of nonzero elements of $F$.

Throughout this work, we will consider RB-operator of weight zero
and assume that $0 \in \mathbb{N}$ and $F$ is a field of characteristic zero.

\begin{definition}
Let $X \neq \emptyset$. 
The set of monomials in the polynomial algebra $F[X]$ is defined as
$$
M(X) = 
\bigg\{ 
\prod_{i \in I} x_i^{\beta_i} \mid 
x_{\alpha} \in X, \ \beta_i \geqslant 0
\bigg\}.
$$
In the product $\prod_{i \in I} x_i^{\beta_i}$, 
only finitely many of the scalars $\beta_i$ are nonzero.

For the algebra $F_0[X]$, we define the set $M_0(X)$ as the subset of $M(X)$ 
consisting of all monomials $\prod_{i \in I} x_i^{\beta_i}$ such that $\sum\limits_{i \in I} \beta_i > 0$.
\end{definition}

\begin{definition}[\!\!\cite{Monom2}]
Let $X \neq \emptyset$.  
A linear operator $L$ on $F[X]$ ($F_0[X]$) is called monomial 
if for every $t \in M(X)$ ($t \in M_0(X)$), there exist  
$z_t \in M(X)$ ($z_t \in M_0(X)$) and $\alpha_t \in F$  
such that $L(t) = \alpha_t z_t$.
\end{definition}

\begin{lemma} [\!\!\cite{GuoMonograph,BGP}] \label{RB_under_automorphism}
Let $A$ be an algebra and let $R$ be an RB-operator of weight $\lambda$ on $A$.

a) The operator $\alpha^{-1} R$ for any $\alpha \in F^*$
is an RB-operator of weight $\alpha\lambda$ on~$A$.

b) The operator  $\psi^{-1} R \psi$ for any $\psi \in \mathrm{Aut}(A)$
is an RB-operator of weight $\lambda$ on~$A$.
\end{lemma}

\begin{lemma} \label{Averaging_under_automorphism}
Let $A$ be an algebra and let $T$ be an averaging operator on $A$.

a) The operator $\alpha^{-1} T$ for any $\alpha \in F^*$ 
is an averaging operator on~$A$.

b) The operator  $\psi^{-1} T \psi$ for any $\psi \in \mathrm{Aut}(A)$
is an averaging operator on~$A$.
\end{lemma}

\begin{definition}
Let $A \in \{ F_0[x, y], F[x, y]\}$.
Define the function $\nu : \mathbb{N} \times A \to \mathbb{N}$
for all $n \in \mathbb{N}$ by the rule $\nu (n, F[x, y]) = 0$ and
$$
\nu(n, F_0[x, y]) = 
\begin{cases}
0, & n > 0, \\
1, & n = 0.
\end{cases}
$$
For simplicity, we will omit the second argument of~$\nu$.
\end{definition}

\begin{lemma} \label{main_lemma_req_seq}
Fix $d, \tau \geqslant 0$ and suppose that there exists a sequence $\{\beta_i\}_{i = \tau}^\infty \in F$
satisfying the recurrence relation
\begin{equation} \label{main_lemma_req_seq_eq}
\beta_s \beta_t = (\beta_s + \beta_t)\beta_{s + t + d}, \quad s, t \geqslant \tau.
\end{equation}
Then one of the following holds: 
either $\beta_i = 0$ for all $i \geqslant \tau$ 
or there exist integers $k$~and~$\Delta$ such that $\beta_k \neq 0$, with
$0 \leqslant k - \tau < \Delta \leqslant k + d$, 
$\Delta \mid k + d$, such that for any $t \geqslant \tau$
one has
\begin{equation} \label{main_lemma_req_seq_solution}
\beta_t = 
\begin{cases}
\dfrac{(k + d)\beta_k}{k + d + \Delta s}, & t = k + \Delta s,\ s \geqslant 0, \\
0, & \mbox{otherwise}.
\end{cases}
\end{equation}
In particular, if $d = \tau = 0$, then $\beta_i = 0$ for all $i \geqslant 0$.
\end{lemma}

\begin{example} \label{main_lemma_req_seq_example}
Suppose that $\beta_i \neq 0$ for all $i \geqslant \tau$ in Lemma~\ref{main_lemma_req_seq}.
Then we have  $k = \tau$, $\Delta = 1$, and the formula~\eqref{main_lemma_req_seq_solution}
takes the form
$\beta_{\tau + s} 
 = \dfrac{(\tau + d)\beta_\tau}{\tau + d + s}$, $s \geqslant 0$.
\end{example}

\begin{lemma} \label{lemma_kratnost}
Let $a, c \in \mathbb{N}_{>0}$. 
If $c \centernot \mid b$ implies $c \centernot \mid a + b$ for all $b > 0$, then $c \mid a$.
\end{lemma}

Note that the notation $\beta^i_j$ refers to a dependence on two indices $i,j$, 
and it is not related to powers of $\beta_j$.

\begin{lemma} \label{main_theorem_req_seq}
Let $N \geqslant 0$, $\{d_s\}_{s = N}^{\infty} \subseteq \mathbb{N}$, 
$\{\tau_s\}_{s = N}^{\infty} \subseteq \mathbb{N}$ be fixed parameters,
and suppose that there exists a sequence 
$\{\beta_i^s \mid i \geqslant \tau_s\}_{s = N}^{\infty} \subseteq F$,
satisfying the recurrence relation
\begin{equation} \label{main_theorem_req_seq_eq}
\beta_m^n \beta_t^s = 
\beta_m^n \beta_{m + t + d_n}^s + \beta_t^s \beta_{m + t + d_s}^n, \
m \geqslant \tau_n, \ t \geqslant \tau_s,\ n, s \geqslant N.
\end{equation}
Then one of the following holds:
either $\beta_t^s = 0$ for all $t \geqslant \tau_s$ and $s \geqslant N$
or there exist set $I \subseteq \mathbb{N}$ and integers $\Delta$, $\{ k_i\}_{i \in I}$, 
such that $\beta_{k_i}^i \neq 0$, with 
$0 \leqslant k_i - \tau_i < \Delta \leqslant k_i + d_i$, 
$\Delta \mid k_i + d_i$, $i \in I$,
such that for all $t \geqslant \tau_s$, $s \geqslant N$,
one has
\begin{equation*}
\beta_t^s = 
\begin{cases}
\dfrac{(k_s + d_s)\beta_{k_s}^s}{k_s + d_s + \Delta l}, & 
s \in I,\ t = k_s + \Delta l,\ l \geqslant 0, \\
0, & \mbox{otherwise}.
\end{cases}
\end{equation*}
In particular, if $d_s = \tau_s = 0$, 
then $\beta_t^s = 0$ for all $t \geqslant \tau_s$.
\end{lemma}

\begin{remark} \label{remark_to_main_theorem_req_seq}
Under the conditions of Lemma~\ref{main_theorem_req_seq}, 
the sequences $\{d_s\}_{s \in S} \subseteq \mathbb{N}$ and 
$\{\tau_s\}_{s \in S} \subseteq \mathbb{N}$ may be considered
for any non-empty subset $\emptyset \neq S \subseteq \mathbb{N}$.
For example, given $0 < N < M$, one may choose 
$S = \{ s \mid N \leqslant s \leqslant M \}$.
\end{remark} 

\begin{lemma} \label{k_i_recurrent_lemma_1}
Let $p, \Delta > 0$ be fixed,
and suppose that there exist two sequences
$\{\xi_{s, t}\}_{s, t = 0}^{\infty} \subseteq \mathbb{N}$ and
$\{k_i\}_{i = 0}^{\infty} \subseteq \mathbb{N}$,
satisfying the recurrence relation
\begin{equation} \label{k_i_recurrent_sequence}
k_{s + t} = k_s + k_t + p - \Delta \xi_{s, t},
\quad s, t \geqslant 0.
\end{equation}
Then $\Delta \xi_{s, 0} = \Delta\xi_{0, s} = k_0 + p$, $s \geqslant 0$, and
\begin{gather} 
k_n = nk_1 + (n - 1)p - \Delta \sum_{s = 1}^{n - 1} \xi_{1, s}, \quad n > 0,
\label{k_i_recurrent_lemma_1_k_n_solution}\\
\xi_{u, v} = -\sum_{s = 1}^{\min\{u, v\} - 1} \xi_{1, s} + 
\sum_{s = \max\{u, v\}}^{u + v - 1} \xi_{1, s}, \quad u, v > 0.
\label{k_i_recurrent_lemma_1_deltas_solution}
\end{gather} 
\end{lemma}

\begin{lemma} \label{k_i_recurrent_lemma_2}
Let $r, p, \Delta > 0$ be fixed, 
and suppose that there exist two sequences 
$\{\xi_{s, t}\}_{s, t = 0}^{\infty} \subseteq \mathbb{N}$ and
$\{k_i\}_{i = 0}^{\infty} \subseteq \mathbb{N}$, satisfying the recurrence relation
\begin{equation} \label{k_i_recurrent_sequence2}
k_{s + t + r} = k_s + k_t + p - \Delta \xi_{s, t},
\quad s, t \geqslant 0.
\end{equation}
Then $\xi_{s, 0} = \xi_{0, s}$, $s \geqslant 0$, and following formulas hold:
\begin{gather}
rk_{ir + j} =
((i + 1)r + j) k_0 + (ir + j)p - 
\Delta \left(
\tau_j + r\sum_{s = 0}^{i - 1}\xi_{0, sr + j} 
\right), \ 0 \leqslant i, \ 0 \leqslant j < r,
\label{k_i_recurrent_sequence2_last} \\
\tau_j = 
j\xi_{0, 0} + j\sum_{s = 1}^{r - 1} (\xi_{0, s + 1} - \xi_{1, s}) -
r\sum_{s = 1}^{j - 1} (\xi_{0, s + 1} - \xi_{1, s}),\ 0 \leqslant j < r;
\label{k_i_recurrent_sequence2_tau_j}
\end{gather}
\begin{multline}\label{k_i_recurrent_sequence2_deltas1}
\xi_{ar + b, cr + d} =
\sum_{s = 1}^{b - 1} (\xi_{0, s + 1} - \xi_{1, s}) + 
\sum_{s = 1}^{d - 1} (\xi_{0, s + 1} - \xi_{1, s}) - 
\sum_{s = 1}^{b + d - 1} (\xi_{0, s + 1} - \xi_{1, s}) \\ + 
\sum_{s = 0}^{a + c} \xi_{0, sr + b + d} - 
\sum_{s = 0}^{a - 1} \xi_{0, sr + b} - 
\sum_{s = 0}^{c - 1} \xi_{0, sr + d},\ b + d < r,
\end{multline}
\begin{multline}\label{k_i_recurrent_sequence2_deltas2}
\xi_{ar + b, cr + d} =
\sum_{s = 1}^{b - 1} (\xi_{0, s + 1} - \xi_{1, s}) + 
\sum_{s = 1}^{d - 1} (\xi_{0, s + 1} - \xi_{1, s}) - 
\sum_{s = 1}^{b + d - r - 1} (\xi_{0, s + 1} - \xi_{1, s}) \\ + 
\sum_{s = 0}^{a + c + 1} \xi_{0, sr + b + d - r} - 
\sum_{s = 0}^{a - 1} \xi_{0, sr + b} - 
\sum_{s = 0}^{c - 1} \xi_{0, sr + d} - 
\xi_{0,0} - \sum\limits_{s=1}^{r - 1} (\xi_{0, s + 1} - \xi_{1, s}),\ b + d \geqslant r,
\end{multline}
where $0 \leqslant a, c, i$, $0 \leq b, d, j < r$. 
\end{lemma}

\section{Linear-in-degree averaging operators}

In~\cite{Khodzitskii2}, the relationship between 
monomial averaging operators and monomial Rota--Baxter operators of nonzero weight on $F[x,y]$ was studied. 
In particular, constructions how to obtain one of these kinds of operators from the other one were proposed. 
We continue these studies by formulating the following problem.
\begin{problem} \label{MainProblem}
Given an operator $T \in \mathcal{T}(X)$ ($\mathcal{T}_0(X)$),
describe all operators $R \in \mathcal{R}(X)$ ($\mathcal{R}_0(X)$)
having the form
\begin{equation*} 
R(z) = \alpha_z T(z), \ z \in M(X) \mbox{ ($M_0(X)$)},\ \alpha_z \in F.
\end{equation*}
\end{problem}
We use the notations $\mathcal{T}(X)$ ($\mathcal{T}_0(X)$) 
and $\mathcal{R}(X)$ ($\mathcal{R}_0(X)$) to denote the sets of 
averaging operators and RB-operators of weight zero on $F[X]$ ($F_0[X]$) respectively.
In the current work, we consider Problem~\ref{MainProblem} for linear-in-degree averaging operators.

A linear-in-degree averaging operator~$T$ is defined by the formula
$T(x^n y^m) = \alpha_{n, m} x^{\alpha n + \beta m + \gamma} y^{p n +q m + c}$ 
for some $\alpha, \beta, \gamma, p, q, c \in \mathbb{N}$ and $\alpha_{n, m} \in F$.
Lemma~3.3 from~\cite{Khodzitskii2} provides us with a method 
for constructing Rota---Baxter operators by an arbitrary linear operator satisfying some certain conditions.
In particular, Corollary~3.3 from~\cite{Khodzitskii2} implies 
that for any averaging operator $T$ on $F_0[x, y]$ or $F[x, y]$ 
such that $1 \notin \Imm T$, there exists an RB-operator~$R$ 
of weight zero defined by the rule
$$
R(x^n y^m) = \begin{cases}
T(x^n y^m) / \deg (T(x^n y^m)), & x^n y^m \notin \ker T, \\
0, & x^n y^m \in \ker T.
\end{cases}
$$
Thus, for any nonzero linear-in-degree averaging operator~$T$ satisfying $1 \notin \Imm T$, 
Problem~\ref{MainProblem} always admits a nontrivial solution.
In the current work, we describe all Rota---Baxter operators of weight zero arising as solutions to Problem~\ref{MainProblem}, 
where~$T$ is a linear-in-degree monomial averaging operator.

Define an automorphism $\psi_{x,y}$ on $F[x,y]$ ($F_0[x,y]$)
by the rule $\psi_{x,y}(x) = y$, $\psi_{x,y}(y) = x$.

\begin{theorem} \label{classification_AvOp}
On $F_0[x, y]$ and on $F[x, y]$, any linear-in-degree monomial averaging operator~$T$, 
up to conjugation by $\psi_{x, y}$, has one of the following forms:

$(\mathrm{i})$ $T(x^n y^m) = \beta_{n, m} x^{r(m + c)} y^{m + c}$, $r \geqslant 0$, $c \geqslant \nu(0)$,

$(\mathrm{ii})$ $T(x^n y^m) = \beta_{n, m} y^{rn + m + c}$, $r, c \geqslant 0$,

$(\mathrm{iii})$ $T(x^n y^m) = \beta_{n, m} x^{n + \gamma} y^{m + c}$, $\gamma, c \geqslant 0$,

$(\mathrm{iv})$ $T(x^n y^m) = \beta_{n, m}$,

\noindent where $\beta_{n,m} \in F$ are some coefficients.
\end{theorem}

\begin{proof}
Consider $T \neq 0$.
Substitute $a = x^n y^m$ and $b = x^s y^t$ into~\eqref{Averaging}, 
assuming that $x^n y^m, x^s y^t \notin \ker T$ and $n + m, s + t \geqslant \nu(0)$.
Then we get a relation on the degrees of the obtained monomials from~\eqref{Averaging}.
Considering the degrees with respect to $x$, we have
\begin{multline*}
\alpha (n + s) + \beta (m + t) + 2\gamma = 
\alpha(\alpha n + \beta m + s + \gamma) + \beta(p n + q m + t + c) + \gamma \\ = 
\alpha(\alpha s + \beta t + n + \gamma) + \beta(p s + q t + m + c) + \gamma.
\end{multline*}
Rewrite the last equality as a system of equations
\begin{gather*}
n(\alpha(\alpha - 1) + \beta p) + m(\alpha\beta + (q - 1)\beta) + \alpha\gamma + \beta c - \gamma = 0,\\
(n - s)(\alpha^2 - \alpha + \beta p)  + (m - t)\beta (\alpha + q - 1) = 0.
\end{gather*}
Joint with the analogous relations for~$y$, we obtain the following system:
\begin{gather}
\begin{gathered} \label{classification_AvOp_main_system}
n(\beta p + \alpha(\alpha - 1)) + 
m(\alpha + q - 1)\beta + (\alpha - 1)\gamma + \beta c = 0,\\
n(\alpha + q - 1)p + m(\beta p + q(q - 1)) + p\gamma + (q - 1)c = 0, \\
(n - s)(\beta p + \alpha(\alpha - 1))  + (m - t) (\alpha + q - 1)\beta = 0,\\
(n - s)(\alpha + q - 1)p  + (m - t)(\beta p + q(q - 1)) = 0.
\end{gathered}
\end{gather}

Denote $A = \{ x^n y^m \mid x^n y^m\notin\ker T \}$.
If $A = \{z\}$, then we show that $T(z) \in F^*$.
This does not hold in the case of $F_0[x, y]$, 
while in the case of $F[x, y]$ we obtain an operator of type~$(\mathrm{iv})$.
Suppose that $T(z) \notin F^*$, then $0 \neq T(z) T(z) = T(T(z)z) = 0$,
it is a~contradiction.
Now consider the case $|A| > 1$.
Note that the following condition holds:
(*) If $|A|>1$, then there are $x^n y^m,x^s y^t\in A$ such that $n \neq s$ or $m \neq t$.

Note that the parameters $\alpha, \beta, \gamma$ and $p, q, c$ 
can be considered up to conjugation of the operator by $\psi_{x, y}$.
For example, if $\alpha = \beta = 0$ and $\gamma > 0$, 
this is equivalent, after conjugation by $\psi_{x, y}$, to the case $p = q = 0$ and $c > 0$.
Taking this observation into account, 
we will henceforth consider cases up to the described symmetry.

{\sc Case 1}.
Let $\alpha + q - 1 = 0$.
Then either $\alpha = 0$ and $q = 1$ or $\alpha = 1$ and $q = 0$.
Suppose that $\alpha = 0$ and $q = 1$.
Then the system~\eqref{classification_AvOp_main_system} can be rewritten as follows:
\begin{gather*}
n\beta p - \gamma + \beta c = 0, \quad
m\beta p + p\gamma= 0, \quad
(n - s)\beta p= 0, \quad
(m - t)\beta p= 0.
\end{gather*}
From~(*) we have $\beta p = 0$.
If $\beta = 0$, then $\gamma = 0$ and $p, c$ are arbitrary.
If $p = 0$, then $\gamma = \beta c$ and $\beta, c$ are arbitrary.
Both described subcases are of the type $(\mathrm{ii})$.

{\sc Case 2}.
Let $\alpha + q - 1 \neq 0$.
There are four subcases: $\alpha = q = 0$, $\alpha = q = 1$, $\alpha > 1$, and $q > 1$.
Due to the symmetry, the case $q > 1$ is equivalent to $\alpha > 1$.

{\sc Case 2.1}.
Let $\alpha = q = 0$. 
Then system~\eqref{classification_AvOp_main_system} takes the following form:
\begin{gather} \label{classification_AvOp_case2.1}
\begin{gathered}
\beta(np - m + c) - \gamma = 0, \quad
p(-n + m\beta + \gamma) - c = 0, \\
\beta((n - s) p - m + t)= 0, \quad
p(-n + s + (m - t)\beta) = 0.
\end{gathered}
\end{gather}

For $\beta = p = 0$, we obtain $\gamma = c = 0$, and thus $T$ has the form $(\mathrm{iv})$.
When $\beta, p \neq 0$, the last two equations take the form
$(n - s)p = m - t$, $n - s = (m - t) \beta$.
By statement~(*), there exist $x^n y^m, x^s y^t \in A$ such that $n \neq s$ or $m \neq t$.
From these relations, it follows that $n \neq s$ and $m \neq t$.
Therefore, $m - t = (n - s)p = (m - t)p \beta$, i.\,e. $\beta = p = 1$. 
From the first two equations of~\eqref{classification_AvOp_case2.1},
we obtain $n = m + \gamma - c$.
Hence, the operator has the form~$(\mathrm{i})$ and is defined by the rule
$$
T(x^n y^m) 
 = \begin{cases}
\varepsilon_{n, m} x^{m + \gamma} y^{m + \gamma}, & n = m + \gamma - c, \\
0, & \text{otherwise}.
\end{cases}
$$

It remains to consider the subcases $\beta \neq 0$, $p = 0$ and $\beta = 0$, $p \neq 0$.
Suppose that $\beta \neq 0$ and $p = 0$. Then  
$c = 0$ by~\eqref{classification_AvOp_case2.1}. Consequently, the relation $\beta m + \gamma = 0$ implies $m = \gamma = 0$,
and hence we obtain an operator of the form~$(\mathrm{iv})$.

{\sc Case 2.2}.
Let $\alpha = q = 1$.
Then the system~\eqref{classification_AvOp_main_system} can be rewritten as follows:
\begin{gather} \label{classification_AvOp_case2.2_system}
\begin{gathered}
\beta(np + m + c) = 0, \quad
p(n + m\beta + \gamma) = 0, \\
\beta((n - s) p  + m - t)= 0, \quad
p(n - s  + (m - t)\beta) = 0.
\end{gathered}
\end{gather}

In the case $\beta = p = 0$, we get an operator of the form $(\mathrm{iii})$: 
$T(x^n y^m) = \varepsilon_{n, m} x^{n + \gamma} y^{m + c}$.
The case $\beta \neq 0$ does not occur,
since~\eqref{classification_AvOp_case2.2_system} holds only for $n = m = 0$,
which contradicts~(*).

The final subcase is $\beta \neq 0$, $p = 0$ or $\beta = 0$, $p \neq 0$.
We consider the case $\beta \neq 0$, $p = 0$,
the case $\beta = 0$, $p \neq 0$ is treated similarly.
Then we obtain $m + c = 0$ and $m - t = 0$ by~\eqref{classification_AvOp_case2.2_system}.
Hence, $c = 0$ and for any $x^a y^b \in A$, we have $b = 0$.
Thus, the operator is defined by the rule
$$
T(x^n y^m) =
\begin{cases}
    \varepsilon_{n, 0} x^{n + \gamma}, & m = 0, \\
    0, & \text{otherwise},
\end{cases}
$$
and is simultaneously of both forms~$(\mathrm{i})$ and~$(\mathrm{iii})$, 
up to conjugation by~$\psi_{x, y}$.

{\sc Case 2.3}.
Let $\alpha > 1$. Then $\alpha(\alpha - 1) > 0$.
If $\beta \neq 0$, the first equation in~\eqref{classification_AvOp_main_system}
implies $n = m = 0$, which contradicts~(*).

Let $\beta = 0$. Then the system~\eqref{classification_AvOp_main_system}
may be rewritten as follows:
\begin{gather*}
n\alpha(\alpha - 1) + (\alpha - 1)\gamma = 0, \quad 
n(\alpha + q - 1)p + mq(q - 1) + p\gamma + (q - 1)c = 0, \\
(n - s)\alpha(\alpha - 1) = 0, \quad 
(n - s)(\alpha + q - 1)p  + (m - t)q(q-1) = 0.
\end{gather*}
From the first relation, we obtain $\gamma = 0$ and $n = 0$.
Hence, $n = 0$ for any $x^n y^m \in A$.
The second and the fourth relations may be reduced to 
$(mq + c)(q - 1) = 0$, $(m - t)q(q-1) = 0$.
From~(*), it follows that $q = 0$ or $q = 1$.
If $q = 0$, then $c = 0$, and we have an operator of the form $(\mathrm{iv})$.
If $q = 1$, then we have an operator of the form $(\mathrm{iii})$.
\end{proof}

\begin{proposition}
Let $T$ be an averaging operator on $F[x, y]$ ($F_0[x, y]$)
that coincides with one of the operators described in Theorem~\ref{classification_AvOp}.
If $\alpha_{n, m} = 1$ for all $n, m$, 
then $T$ is an example of an operator corresponding to one of cases~$(\mathrm{i})$--$(\mathrm{iii})$.
An operator of the form~$(\mathrm{iv})$ is realized for arbitrary $\alpha_{n, m} \in F$
in the case of $F[x, y]$, and for $\alpha_{n, m} = 0$,
whenever $n + m \geqslant 1$, i.\,e. $T = 0$, in the case of $F_0[x, y]$.
\end{proposition}

\begin{example}
On algebra $F[x,y]$, all monomial homomorphic averaging operators 
are linear-in-degree (see Theorem 3.1 in~\cite{Khodzitskii}). 
However, there exist monomial averaging operators on $F[x, y]$
that are not linear-in-degree.
For example, consider the operator $T$ defined as $T(x^{an + b} y^m) = x^{an} y^{m + b}$,
where $0 \leqslant a, m$, $0 \leqslant b < n$ and $0 < n$ is a fixed integer.
\end{example}

\begin{remark}
Note that a solution to Problem~\ref{MainProblem} always exists
for averaging operators from Theorem~\ref{classification_AvOp},  
assuming $\beta_{n, m} = 1$ for all $n, m \geqslant 0$.
Indeed, let $T$ be a~monomial averaging operator 
on $F[X]$ ($F_0[X]$), and let $R$ be an RB-operator 
defined by $R(z) = \alpha_z T(z)$.
Then $P(z) = T(z) / \beta_z$ is an averaging operator,
where $T(z) = \beta_z w$ for all $z \notin\ker T \cap M(X)$.
In this case, we have $R(z) = \varepsilon_z P(z)$ with $\varepsilon_z = \alpha_z \beta_z$
and the operator $P$ maps each monomial outside the kernel to a monomial with coefficient~1.
\end{remark}

\section{Cases $(\mathrm{i})$ and $(\mathrm{ii})$}

We consider both Cases $(\mathrm{i})$ and $(\mathrm{ii})$ in the current paragraph,
since Lemma~\ref{main_theorem_req_seq} plays a crucial role 
in solving Problem~\ref{MainProblem} 
for operators of forms $(\mathrm{i})$ and $(\mathrm{ii})$, as outlined in Theorem~\ref{classification_AvOp}.

\subsection{Case $(\mathrm{ii})$}

\begin{theorem} \label{Case2_theorem}
Let $R$ be an RB-operator of weight zero on $F_0[x, y]$ or $F[x, y]$,
defined by $R(x^n y^m) = \alpha_{n,m} y^{m + r n + c}$, 
where~$r \geqslant 0$ and $c \geqslant \nu(0)$.

a) If $r = c = 0$, then
on $F[x, y]$ we have $R = 0$,
and on $F_0[x, y]$ there exists $k > 0$ such that
$\alpha_{0, k} \neq 0$, and for all $n + m \geqslant 1$ 
the coefficients $\alpha_{n,m}$ satisfy
\begin{equation} \label{RBO_Case2a_solution}
\alpha_{n, m} = 
\begin{cases}
\alpha_{0, k} / a, & n = 0, \ m = a k, \ a \geqslant 0, \\
0, & \mbox{otherwise}.
\end{cases}
\end{equation}

b) If $r + c > 0$, then there exist set $I \subseteq\mathbb{N}$
and set of parameters $\Delta$, $\{ k_i\}_{i \in I}$, 
where $\alpha_{i, k_i} \neq 0$, $\nu(i) \leqslant k_i$, 
and $k_i - \nu(0) < \Delta \leqslant k_i + ri + c$,  with
$\Delta \mid k_i + ri + c$, such that for all $l + t \geqslant \nu(0)$ 
the coefficients $\alpha_{n,m}$ satisfy
\begin{equation} \label{RBO_Case2b_solution}
\alpha_{l, t} = 
\begin{cases}
\dfrac{(k_l + rl + c)\alpha_{l, k_l}}{k_l + rl + c + \Delta s}, 
& l \in I,\ t = k_l + \Delta s, \ s \geqslant 0, \\
0, & \mbox{otherwise}.
\end{cases}
\end{equation}
\end{theorem}

\begin{proof}
Define a~set $I \subseteq \mathbb{N}$ as follows:
$$
I = \{ i \mid \mbox{there is } j \mbox{ such that } \alpha_{i, j} \neq 0\}.
$$

Substituting $a = x^n y^n$ and $b = x^s y^t$ into~\eqref{RBO}, 
we obtain the following relations:
\begin{equation} \label{RBO_Case2}
\alpha_{n,m}\alpha_{s,t} = 
\alpha_{n,m}\alpha_{s, rn + m + t + c} + \alpha_{s,t}\alpha_{n, rs + m + t + c}, 
\quad n + m, s + t \geqslant \nu(0).
\end{equation}

Denote $\beta_m^n = \alpha_{n, m}$, $d_n = rn + c$, $n + m \geqslant \nu(0)$.
Then~\eqref{RBO_Case2} takes the form:
$$
\beta_m^n \beta_t^s = \beta_m^n \beta_{m + t + d_n}^s + \beta_t^s \beta_{m + t + d_s}^n,
\quad n \geqslant \tau_n, \ s \geqslant \tau_s,
$$
where $\tau_i = \nu(i)$. For the parameters defined above, 
we can apply Lemma~\ref{main_theorem_req_seq} with $N = 0$.

If $r = c = 0$, then $d_n = 0$, $n \geqslant 0$.
Furthermore, on $F[x, y]$, we have $\tau_n = 0$ 
for all $n \geqslant 0$, implying $R = 0$.
On $F_0[x, y]$, we have $\tau_n = 0$, $n > 0$, and $\tau_0 = \nu(0) = 1$.
Therefore, if $R \neq 0$, there exists a constant $k_0 \geqslant \nu(0)$
such that for any $a \geqslant 0$ the following formula holds:
$$
\alpha_{0, s} = 
\begin{cases}
k_0\alpha_{0, k_0} / (k_0 + a k_0), & s = k_0 + a k_0, \\
0, & \mbox{otherwise}.
\end{cases}
$$
This leads to the formula given in~\eqref{RBO_Case2a_solution}.

When $r + c > 0$, we directly obtain the formula~\eqref{RBO_Case2b_solution} 
applying Lemma~\ref{main_theorem_req_seq}.
\end{proof} 

\begin{example}
Let $I \subset \mathbb{N}$ be an arbitrary set and $r + c > 0$.
Then, substituting $\Delta = 1$ and $k_i = \nu(i)$ for all $i \in I$
in Theorem~\ref{Case2_theorem},
the formula~\eqref{RBO_Case2b_solution} takes the following form:
$$
\alpha_{i, t} =
\begin{cases}
\dfrac{(ri + c)\alpha_{i, \nu(i)}}{t + ri + c}, & i \in I, \ t \geqslant \nu(i), \\
0, & \mbox{otherwise}.
\end{cases}
$$
\end{example}

\begin{example}
Suppose that $\Delta = 8$, $r = 1$, $c = 8d$, $d > 0$, 
and $I = \{ 8p + 2q \mid 1 \leqslant q \leqslant 4,\ 0 \leqslant p\}$.
For any $8p + 2q \in I$ define $k_{8p + 2q} = 8 - 2q$. 
Then, we have:
$$
k_{8p + 2q} = 8 - 2q < 8 \leqslant k_i + ri + c = 8 - 2q + 1(8p + 2q) + 8d = 8(p + 1 + d),
$$
which shows that $\Delta \mid k_i + ri + c$.
Thus, we can rewrite the formula~\eqref{RBO_Case2b_solution} as:
$$
\alpha_{i, t} = 
\begin{cases}
\dfrac{(p + d + 1)\alpha_{8p + 2q, 8 - 2q}}{p + d + s + 1}, 
& i = 8p + 2q, \ t = 8 - 2q + 8s,\  1 \leqslant q \leqslant 4,\ 0 \leqslant p, s, \\
0, & \mbox{otherwise}.
\end{cases}
$$
\end{example}

\subsection{Case $(\mathrm{i})$}

Define a~function $\zeta_r \colon \mathbb{Z} \to \mathbb{N}$ for 
$r \geqslant 0$, by the following rule:
$$
\zeta_r(i) = 
\begin{cases}
- \lceil \min\{i / r, (i - \nu(0)) / (r + 1)\}\rceil, & i < 0, \\
\nu(0), & i = 0, \\
0, & i > 0.
\end{cases}
$$
The ceiling function $\lceil x \rceil$ denotes the smallest integer greater than or equal to $x$.

\begin{theorem} \label{Case1_theorem}
Let $R$ be an RB-operator of weight zero on $F_0[x, y]$ or $F[x, y]$
defined by $R(x^n y^m) = \alpha_{n,m} x^{r (m + c)} y^{m + c}$, 
where $r \geqslant 0$ and $c \geqslant \nu(0)$.

a) If $r = c = 0$, then
on $F[x, y]$ we have $R = 0$ on $F[x, y]$,
and on $F_0[x, y]$ there exists $k > 0$ such that $\alpha_{0, k} \neq 0$,
and for all $n + m \geqslant 1$ the coefficients $\alpha_{n, m}$ satisfy:
\begin{equation*}
\alpha_{n, m} = 
\begin{cases}
\alpha_{0, k} / a, & n = 0, \ m = a k, \ a \geqslant 0, \\
0, & \mbox{otherwise}.
\end{cases}
\end{equation*}
 
b) If $r + c > 0$, then there exist set $I\subseteq\mathbb{Z}$ 
and set of parameters $\Delta$, $\{ k_i\}_{i \in I}$, where $\alpha_{rk_i + i, k_i} \neq 0$,
and $0 \leqslant k_i - \zeta_r(i) < \Delta \leqslant k_i + c$, with $\Delta \mid k_i + c$,
such that for all $(r + 1)l + t \geqslant \nu(0)$, $l \in \mathbb{N}$, $t \in \mathbb{Z}$,
the coefficients $\alpha_{rl + t, l}$ satisfy:
\begin{equation} \label{RBO_Case1b_solution}
\alpha_{rl + t, l} = 
\begin{cases}
\dfrac{(k_t + c)\alpha_{rk_t + t, k_t}}{k_t + c + \Delta s}, 
& t \in I, \ l = k_t + \Delta s, \ s \geqslant 0,\\
0, & \mbox{otherwise}.
\end{cases}
\end{equation}
\end{theorem}

\begin{proof}
Let us define a~set $I \subseteq \mathbb{Z}$ as follows:
\begin{equation*}
\begin{gathered}
I = \{ i \mid \mbox{there is } j \geqslant 0 \mbox{ such that } \alpha_{rj + i, j} \neq 0, 
\ (r + 1)j + i \geqslant \nu(0)\}.
\end{gathered}
\end{equation*}

Substituting $a = x^u y^v$ and $b = x^p y^q$ into~\eqref{RBO} yields the following relation:
\begin{equation} \label{RBO_Case1}
\alpha_{u, v}\alpha_{p, q} = 
\alpha_{u, v}\alpha_{r(v + c) + p, v + q + c} + 
\alpha_{p, q}\alpha_{r(q + c) + u, v + q + c},
\ u + v, p + q \geqslant \nu(0).
\end{equation}
Let us introduce the notation $\beta_a^b = \alpha_{ra + b, a}$, 
where $(r + 1)a + b\geqslant\nu(0)$, $b \geqslant -ra$, and $a \geqslant 0$.
Using this notation, equality~\eqref{RBO_Case1} has the form
\begin{gather} \label{RBO_Case1_betas}
\begin{gathered}
\beta_n^m \beta_s^t = \beta_n^m \beta_{n + s + c}^t + \beta_s^t \beta_{n + s + c}^m, \\
(r + 1)n + m, (r + 1)s + t \geqslant \nu(0),\quad
t \geqslant -rs,\quad m \geqslant -rn,\quad n, s \geqslant 0.
\end{gathered}
\end{gather}

a) Let $r = 0$.
In this case, the operator $R$ acts on $x^n y^m$ as follows:
$R(x^n y^m) = \alpha_{n,m} x^{r (m + c)} y^{m + c} = \alpha_{n,m} y^{m + c}$,
which means that $R$ is an operator of the form $(\mathrm{ii})$ from Theorem~\ref{classification_AvOp}.
All these operators are described in Theorem~\ref{Case2_theorem}.

b) Let $r > 0$.
Consider~\eqref{RBO_Case1_betas} for $t, m > 0$ and $t, m \leqslant 0$ separately.

For $t, m > 0$, we apply Lemma~\ref{main_theorem_req_seq} to the
subsequence of equality~\eqref{RBO_Case1_betas} 
with $N = 1$, $d_s = c$ and $\tau_s = 0$ for all $s \geqslant N$.
If $c = 0$, the recurrent sequence is trivial.
Suppose $c > 0$. Then there exists $\Delta_1 > 0$, and for all $i \in I \cap \mathbb{Z}_{>0}$
there exists $k_i \geqslant 0$, $\alpha_{rk_i + i, k_i} \neq 0$,
such that $k_i < \Delta_1 \leqslant k_i + c$, $\Delta_1 \mid k_i + c$, 
and for all $(r + 1) + i \geqslant \nu(0)$ the following formula holds:
\begin{equation} \label{RBO_Case1_betas_solution}
\alpha_{rl + i, l} = 
\begin{cases}
\dfrac{(k_i + c)\alpha_{rk_i + i, k_i}}{k_i + c + \Delta_1 s}, 
& i \in I \cap \mathbb{Z}_{>0},\ l = k_i + \Delta_1 s, \ s \geqslant 0, \\
0, & \mbox{otherwise}.
\end{cases}
\end{equation}

Now, consider the equality~\eqref{RBO_Case1_betas} for $t, m \leqslant 0$.
If $m = t$ for a fixed $t < 0$, then the constraints listed in~\eqref{RBO_Case1_betas} take the form
$n, s \geqslant -\lceil \min\{t / r, (t - \nu(0)) / (r + 1)\}\rceil$, $t < 0$.
For $t = 0$, we have $n, s \geqslant \nu(0)$.
Thus, $n, s \geqslant \zeta_r(t)$ holds for all $t \leqslant 0$,
where $\zeta_r(t)$ is the function defined earlier.
Denote $\gamma_s^{-t} = \beta_s^t$, $t \leqslant 0$.
The relation~\eqref{RBO_Case1_betas} holds for the sequence $\gamma_s^{-t}$.
Taking into account Remark~\ref{remark_to_main_theorem_req_seq}, 
we can apply Lemma~\ref{main_theorem_req_seq} 
for $t, m \leqslant 0$, with $N = 0$, $d_s = c$ and $\tau_s = \zeta_r(s)$, $s \geqslant 0$.
Thus, there exists $\Delta_2 > 0$, and for all $i \in I \cap \mathbb{Z}_{\leqslant 0}$
there exist $k_i \geqslant \zeta_r(i)$, $\alpha_{rk_i - i, k_i} \neq 0$, 
such that $k_i - \zeta_r(i) < \Delta_2 \leqslant k_i + c$, 
$\Delta_2 \mid k_i + c$, and for all $(r + 1)l - t \geqslant \nu(0)$
the following formula holds:
\begin{equation} \label{RBO_Case1_betas2_solution}
\alpha_{rl - t, l} = 
\begin{cases}
\dfrac{(k_t + c)\alpha_{rk_t - t, k_t}}{k_t + c + \Delta_2 s}, 
& i \in I \cap \mathbb{Z}_{\leqslant 0}, \ l = k_t + \Delta_2 s, \ s \geqslant 0, \\
0, & \mbox{otherwise}.
\end{cases}
\end{equation}

Note that combining the formulas~\eqref{RBO_Case1_betas_solution} and~\eqref{RBO_Case1_betas2_solution},
we obtain~\eqref{RBO_Case1b_solution}.
Additionally, in the case $c = 0$, the formula~\eqref{RBO_Case1_betas_solution} 
describes a trivial recurrence sequence, since $I \cap \mathbb{Z}_{>0} = \emptyset$.

Let $I_{>0} = I \cap \mathbb{Z}_{>0}$ and $I_{\leqslant 0} = I \cap \mathbb{Z}_{\leqslant 0}$.
It remains to verify that~\eqref{RBO_Case1} holds for all $\alpha_{rm + n, m}$, $\alpha_{rt + s, t}$,
where $n \in I_{>0}$ and $s \in I_{\leqslant 0}$.
It is clear that $I = I_{\leqslant 0} \cup I_{> 0}$.
Equality~\eqref{RBO_Case1} already holds on $I_{\leqslant 0}$ and $I_{> 0}$ 
by~\eqref{RBO_Case1_betas_solution} and~\eqref{RBO_Case1_betas2_solution} respectively.
Thus, it remains to verify that~\eqref{RBO_Case1} holds when both formulas are applied simultaneously.

Let $I_{>0} \neq \emptyset$ and $\alpha_{r(k_i + \Delta_1 a) + i, k_i + \Delta_1 a} \neq 0$,
$i \in I_{>0}$, $a \geqslant 0$.
If $I_{\leqslant 0} = \emptyset$, then $\alpha_{rs - j, s} = 0$ 
for all $0 < s$, $0 \leqslant j \leqslant rs$.
If $I_{\leqslant 0} \neq \emptyset$, then we consider
$\alpha_{rs - j, s} = 0$, $j \in I_{\leqslant 0}$, $s \neq k_j + \Delta_2 b$, $b \geqslant 0$.
In the both cases we obtain
\begin{equation*}
\alpha_{r(k_i + \Delta_1 a) + i, k_i + \Delta_1 a}\alpha_{rs - j, s} = 0 =
\alpha_{r(k_i + \Delta_1 a) + i, k_i + \Delta_1 a}
\alpha_{r(k_i + \Delta_1 a + s + c) - j, k_i + \Delta_1 a + s + c}.
\end{equation*}
If $I_{\leqslant 0} = \emptyset$, then $\alpha_{r(k_i + \Delta_1 a + s + c) - j, k_i + \Delta_1 a + s + c} = 0$,
since $\alpha_{rd - j, d} = 0$ for all $d \geqslant 0$.
A similar argument holds in the case $I_{>0} = \emptyset$ and $I_{\leqslant 0} \neq \emptyset$.
For $I_{\leqslant 0} \neq \emptyset$, we have
$\alpha_{r(k_i + s + c +\Delta_1 a) - j, k_i + s + c + \Delta_1 a} = 0$.
From this, we obtain the following relations:
$$
\Delta_2 \centernot \mid s - k_j, \quad
\Delta_2 \centernot \mid s - k_j + k_i + c + \Delta_1 a, \ a \geqslant 0.
$$
Subsequently substituting $a = 0$ and $a = 1$, we apply Lemma~\ref{lemma_kratnost}
to conclude that $\Delta_2 \mid k_i + c$ and $\Delta_2 \mid \Delta_1$ respectively.
Analogously dealing with $i \in I_{\leqslant 0}$ and $j \in I_{>0}$,
we deduce that $\Delta_1 \mid \Delta_2$.
Thus, we have $\Delta_1 = \Delta_2 = \Delta$.
Moreover, the formulas~\eqref{RBO_Case1_betas_solution} and~\eqref{RBO_Case1_betas2_solution}
imply~\eqref{RBO_Case1b_solution}.

The case where $\alpha_{r(k_i + \Delta a) + i, k_i + \Delta a} \neq 0$ and
$\alpha_{r(k_j + \Delta b) - j, k_j + \Delta b} \neq 0$
may be considered analogously to Case~2 of Lemma~\ref{main_theorem_req_seq}. 
The detailed calculations in this case are omitted for brevity.
Thus, the proof of the theorem is complete.
\end{proof}

Consider the geometrical interpretation of the formula~\eqref{RBO_Case1b_solution}.
\begin{center}
\pgfplotsset
{
width = 14cm, 
height = 8cm, 
compat = 1.18
}
\begin{tikzpicture}
\tikzstyle{every node}=[font=\small]
\begin{axis}
[
    axis x line = middle,
    axis y line = middle,
    xtick={1500},
    ytick={1500},
    xmin = -7,
    ymin = 0,
    xmax = 35,
    ymax = 49,
    grid = none
]
\addplot [thick, black]
coordinates 
{
    (0, 0)
    (36, 50)
};
\addplot [thick, lightgray]
coordinates 
{
    (6, 0)
    (36, 41.67)
};
\addplot [very thick, gray]
coordinates 
{
    (12, 0)
    (36, 33.34) 
};
\addplot [thick, lightgray]
coordinates 
{
    (18, 0)
    (36, 25.01)
};
\addplot [thick, lightgray]
coordinates 
{
    (24, 0)
    (36, 16.68)
};
\addplot [thick, lightgray]
coordinates 
{
    (30, 0)
    (36, 8.35)
};
\addplot [very thick, gray]
coordinates 
{
    (-6, 0)
    (30, 50)
};
\addplot [thick, lightgray]
coordinates 
{
    (-12, 0)
    (24, 50)
};
\addplot [thick, lightgray]
coordinates 
{
    (-18, 0)
    (18, 50)
};
\addplot [thick, lightgray]
coordinates 
{
    (-24, 0)
    (12, 50)
};
\addplot [thick, lightgray]
coordinates 
{
    (-30, 0)
    (6, 50)
};
\addplot [thick, dashed, black]
coordinates 
{
    (-7, 15)
    (36, 15)
};

\draw[>=latex,->,font=\scriptsize] (4.8, 15.2) to [out=70,in=165] (10, 22.7);
\draw[>=latex,->,font=\scriptsize] (10.2, 22.5) to [out=70,in=165] (15.45, 30.28);
\draw[>=latex,->,font=\scriptsize] (15.6, 30) to [out=70,in=165] (20.8, 37.6);
\draw[>=latex,->,font=\scriptsize] (21, 37.5) to [out=70,in=165] (26.2, 45.2);
\draw[>=latex,->,font=\scriptsize] (26.4, 45) to [out=70,in=165] (31.6, 52.7);

\draw[>=latex,->,font=\scriptsize] (15.92, 5.5) to [out=70,in=165] (21.12, 13.2);
\draw[>=latex,->,font=\scriptsize] (21.32, 13) to [out=70,in=165] (26.52, 20.7);
\draw[>=latex,->,font=\scriptsize] (26.72, 20.5) to [out=70,in=165] (31.92, 28.4);
\draw[>=latex,->,font=\scriptsize] (31.92, 28) to [out=70,in=165] (37.52, 35.9);
\addplot
[
    mark=*,
    ForestGreen!75!white,
    mark size=2pt,
    only marks,
    point meta=explicit symbolic
]
coordinates 
{
    (15.92, 5.5)
    (21.32, 13)
    (26.72, 20.5)
    (32.12, 28)
    (37.52, 35.5)
};
\addplot
[
    mark=*,
    blue!65!white,
    mark size=2pt,
    only marks,
    point meta=explicit symbolic
]
coordinates 
{
    (4.8, 15)
    (10.2, 22.5)
    (15.6, 30)
    (21, 37.5)
    (26.4, 45)
};

\node [right, black] at (30, 41) {\begin{turn}{29}$\overrightarrow{(r, 1)}$\end{turn}};
\node [right, black] at (5, 6) {\begin{turn}{29}$L_0$\end{turn}};
\node [right, black] at (1, 9) {\begin{turn}{29}$L_u$\end{turn}};
\node [right, black] at (14, 2) {\begin{turn}{29}$L_v$\end{turn}};
\node [right, red] at (6.9, 11.9) {$u$};
\node [right, red] at (16.1, 18.5) {$v$};

\node [right, blue!65!white] at (4.2, 13.2) {$\begin{turn}{27}$k_u$\end{turn}$};
\node [right, blue!65!white] at (8.8, 21)  {\begin{turn}{27}$k_u +  \Delta$\end{turn}};
\node [right, blue!65!white] at (14.2, 28.7) {\begin{turn}{27}$k_u + 2 \Delta$\end{turn}};
\node [right, blue!65!white] at (19.6, 36.3) {\begin{turn}{27}$k_u + 3 \Delta$\end{turn}};
\node [right, blue!65!white] at (24.95, 43.8) {\begin{turn}{27}$k_u + 4 \Delta$\end{turn}};

\node [right, ForestGreen!75!white] at (15.32, 3.5) {$\begin{turn}{27}$k_v$\end{turn}$};
\node [right, ForestGreen!75!white] at (19.92, 11.5)  {\begin{turn}{27}$k_v +  \Delta$\end{turn}};
\node [right, ForestGreen!75!white] at (25.32, 19.2) {\begin{turn}{27}$k_v + 2 \Delta$\end{turn}};
\node [right, ForestGreen!75!white] at (30.72, 26.8) {\begin{turn}{27}$k_v + 3 \Delta$\end{turn}};

\node [right, Black] at (11.3, 2) {\begin{turn}{5}$v$\end{turn}};
\node [right, Black] at (-1.5, 2) {\begin{turn}{5}$0$\end{turn}};
\node [right, Black] at (-6.9, 2) {\begin{turn}{5}$u$\end{turn}};

\draw [decorate, decoration={brace,amplitude=6pt}, red, thick] (10.7, 15) -- (22.7, 15);
\draw [decorate, decoration={brace,amplitude=6pt,mirror}, red, thick] (4.7, 14.8) -- (10.7, 14.8);
\end{axis}
\end{tikzpicture}

Fig.~1.
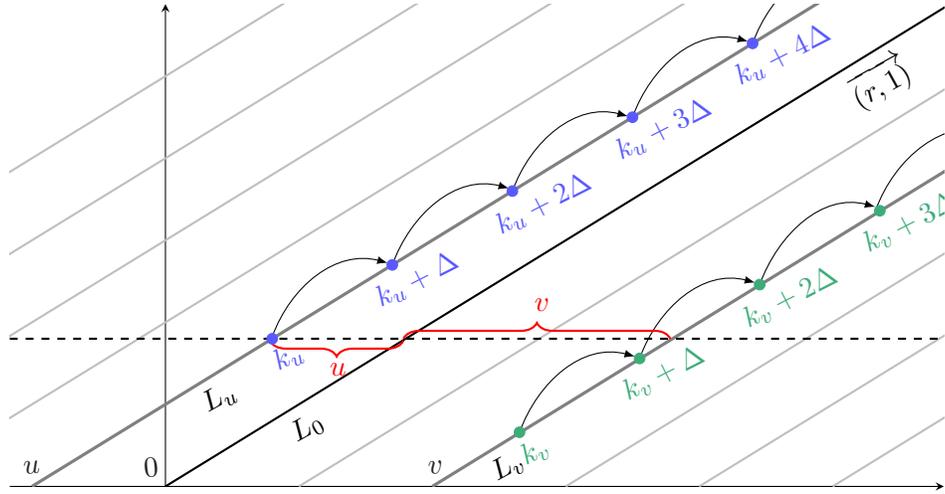
The position of the points corresponding to the indices of nonzero $\alpha_{i,j}$ from~\eqref{RBO_Case1b_solution}
on the coordinate plane
\end{center}

Consider the integer lattice on the coordinate plane.
For any $i \in \mathbb{Z}$, define a line $L_i$
be the line passing through the point $(i,0)$ 
with the direction vector $(r, 1)$.
Each of nonzero coefficients among $\alpha_{a,b}$ corresponds to a point $(rj + i, j)\in L_j$. 
Fix $i \in \mathbb{Z}$. 
The point lying on the line $L_i$ with the smallest $x$ coordinate has the $y$ coordinate $j = k_i$;
all subsequent points on $L_i$ are separated from each other along the $x$~axis at the distance of~$\Delta$.

\begin{example}
Consider $F[x, y]$. 
Let the parameters from Theorem~\ref{Case1_theorem} be chosen as follows:
$\Delta = r = c = 2$ and $I = 2\mathbb{Z}$. 
Then $k_{2i} = \zeta_2(2i)$, i.\,e. $k_{2i} = -i$ for $i < 0$ and $k_{2i} = 0$ for $i \geqslant 0$.
So the formula~\eqref{RBO_Case1b_solution} takes the form
$$
\alpha_{rl + t, l} = 
\begin{cases}
\dfrac{(u + 2)\alpha_{4u, u}}{u + 2(s + 1)}, & -t = 2u, \ l = u + 2s, \ u > 0, \ s \geqslant 0, \\
\dfrac{2\alpha_{2u, 0}}{2(s + 1)}, & t = 2u,\ l = 2s,\ u, s \geqslant 0, \\
0, & \mbox{otherwise}.
\end{cases}
$$
\end{example}

\section{Case $(\mathrm{iii})$}

We consider Case~$(\mathrm{iii})$ step by step.
For an RB-operator $R(x^n y^m) = \alpha_{n, m} x^{n + p} y^{m + q}$,
we distinguish two subcases: $p = q = 0$ and $p + q > 0$.
The case $p = q = 0$ was completely considered in~\cite{Khodzitskii2}.
The case $p + q > 0$ will be considered first in the context of averaging operators. 
Then we will describe all RB-operators constructed from the obtained averaging operators.

\subsection{Case $p = q = 0$}

\begin{theorem}[\!\!\cite{Khodzitskii2}] \label{RBO_Case3_p=0_q=0_theorem}
Let $R$ be an RB-operator of weight zero on $F_0[x, y]$ or $F[x, y]$ defined as
$R(x^n y^m) = \alpha_{n, m} x^n y^m$, $\alpha_{n, m} \in F$, $n, m \geqslant 0$.

a) On $F[x, y]$, there exists only trivial RB-operator $R = 0$.

b) On $F_0[x, y]$, either there exists $\gamma \in F$ such that
for any $s,t \geqslant 0$, $s + t > 0$, we have
\begin{gather} \label{RBO_Case3_1_F0_solution_first}
\alpha_{s, t} =
\begin{cases}
    \gamma / a, & s = al, \, t = ar, \, a > 0, \\
    0, & \text{otherwise},
\end{cases}
\end{gather}
or there exist $k_1, k_2 > 0$, $\alpha_1, \alpha_2 \neq 0$,
and some $a, b ,d\geqslant 0$ such that
$d = \gcd (k_1, k_2)$, $a < d$, $b \mid d$, $d/b \mid \gcd (a, d)$, 
and for all $s, t \geqslant 0$, $s + t > 0$, we have
\begin{gather} \label{RBO_Case3_1_F0_solution_second}
\alpha_{s, t}
 = \begin{cases}
 \dfrac{\alpha_1\alpha_2}
 {(t / k_2)\alpha_1 + (s / k_1)\alpha_2},
 & s \equiv_{k_1} (k_1/b) l,\,t \equiv_{k_2} (k_2a/d)l,\, 0 \leqslant l < b,\\
 0, & \mbox{otherwise}.
\end{cases}
\end{gather}
\end{theorem}

\subsection{Case $p + q > 0$ for averaging operators}

In the next theorem, to avoid duplication of notation, 
certain variables will be represented by a single symbol. 
This symbol can be substituted by different variables depending on the context. 
For example, $\theta^\lambda$, where $\lambda \in \{ x, y\}$,
$\lambda$ can be replaced with either $x$ or $y$.
We also suppose that $\sum\limits_{i = a}^b \gamma_i = 0$
if $b \geqslant a$ does not hold.

\begin{theorem} \label{RBO_Case3_main_theorem}
Let $T$ be an averaging operator on $F[x,y]$ or $F_0[x,y]$
defined by $T(x^n y^m) = \alpha_{n,m} x^{n + p_x} y^{m + p_y}$, 
where $p_x > 0$ and $p_y \geqslant 0$.
Then, up to multiplication by a nonzero scalar,
exactly one of the following cases holds:

a) There exist parameters $k, c \geqslant \nu(0)$ with $\alpha_{k, c} \neq 0$ and $\Delta > 0$,
such that $c \leqslant \Delta \leqslant c + p_y$, 
$\Delta \mid c + p_y$, and for all $n + m \geqslant \nu(0)$
the coefficients $\alpha_{n, m}$ satisfy
\begin{equation} \label{RBO_Case3a_main_theorem_solution}
\alpha_{n, m} =
\begin{cases}
1, & n = k + \frac{k + p_x}{c + p_y}\Delta s, \ m = c + \Delta s, \  s \geqslant 0,\\
0, & \mbox{otherwise}.
\end{cases}
\end{equation}

b) If $p_y > 0$, then there exist parameters 
$c_x, c_y, r_x, r_y, \Delta_x, \Delta_y > 0$, $k_0 \geqslant \nu(c_y)$,
and sequences $\{ \sigma_{0, l}\}_{l = 1}^\infty$, $\{ \sigma_{1, l}\}_{l = 1}^{r_y - 1}$,
satisfying $r_x \Delta_x = c_x + p_x$, $r_y \Delta_y = c_y + p_y$,
and moreover there exist sequences 
$\{ k_l\}_{l = 0}^\infty, \{ \tau_j\}_{j = 0}^{r_y - 1} \subseteq \mathbb{N}$ such that 
$\alpha_{k_l, l} \neq 0$, $0 \leqslant k_l - \nu(l) < \Delta_x$,
\begin{gather*}
r_y k_{ir_y + j} = ((i + 1)r_y + j) k_0 + (i r_y + j)p_x - 
\Delta_x(\tau_j + \sum_{s = 0}^{i - 1}r_y \sigma_{0, s r_y + j}),\\
\tau_j = j\sigma_{0, 0} + 
\sum_{s = 1}^{r_y - 1} j(\sigma_{0, s + 1} - \sigma_{1, s}) - 
\sum_{s = 1}^{j - 1} r_y(\sigma_{0, s + 1} - \sigma_{1, s}),
\end{gather*}
and for all $s + t \geqslant \nu(0)$ the coefficients $\alpha_{n, m}$ satisfy
\begin{equation} \label{RBO_Case3b_main_theorem_solution}
\alpha_{s, t} =
\begin{cases}
1, & s = k_j + \Delta_x i, \ t = c_y + \Delta_y j, \ i, j \geqslant 0, \\
0, & \mbox{otherwise}.
\end{cases}
\end{equation}

If $p_y = 0$, then there exist parameters $c, \Delta > 0$, $k_0 \geqslant \nu(0)$,
$k_1 \geqslant 0$, and a sequence $\{ \sigma_i\}_{i = 1}^\infty$, 
satisfying the conditions $\alpha_{k_0, 0}, \alpha_{k_1, c} \neq 0$,
$k_0 - \nu(0) < \Delta \leqslant k_0 + p_x$, $\Delta \mid k_0 +p_x$,
and there exists a sequence $\{ k_n\}_{n = 1}^\infty$ such that
$k_n = n k_1 + (n - 1)p_x - \Delta \sum_{s = 1}^{n - 1} \sigma_s$
and for all $n + m \geqslant \nu(0)$ the coefficients $\alpha_{n, m}$ satisfy
\begin{equation} \label{RBO_Case3b_main_theorem_solution_p_y=0}
\alpha_{n, m} =
\begin{cases}
1, & n = k_s + \Delta l, \ m = sc, \ s, l \geqslant 0,\\
0. & \mbox{otherwise}.
\end{cases}
\end{equation}
\end{theorem}

\begin{proof}
Substituting $a = x^s y^t$ and $b = x^n y^m$ into~\eqref{Averaging}, 
we obtain the following system of relations on the coefficients:
\begin{equation} \label{RBO_Case3_3_averaging}
\alpha_{n, m} \alpha_{s, t} = 
\alpha_{n, m} \alpha_{n + s + p_x, m + t + p_y} = 
\alpha_{s, t} \alpha_{n + s + p_x, m + t + p_y},\ n + m, s + t \geqslant \nu(0).
\end{equation}
From the above equality, assuming $\alpha_{s, t}, \alpha_{n, m} \neq 0$, 
we deduce that $\alpha_{s, t} = \alpha_{n, m}$.
Therefore, there exists a constant $\gamma \neq 0$ such that if $\alpha_{s, t} \neq 0$, 
then $\alpha_{s, t} = \gamma$ for all $s + t \geqslant \nu(0)$.
Without loss of generality, we set $\gamma = 1$ by Lemma~\ref{Averaging_under_automorphism}.
If $\alpha_{n, m} = 0$ in~\eqref{RBO_Case3_3_averaging}, 
then $\alpha_{s, t} = 0$ or $\alpha_{n + s + p_x, m + t + p_y} = 0$.
Furthermore, applying~\eqref{RBO_Case3_3_averaging} for $\alpha_{n, m} \neq 0$ several times,
we obtain the following formulas:
\begin{gather}
\alpha_{s + a(n + p_x), t + a(m + p_y)} = \alpha_{s ,t}, \quad a \geqslant 0, 
\label{RBO_Case3_3_averaging_eq1} \\
\alpha_{(a + 1)n + ap_x, (a + 1)m + ap_y} = \alpha_{n, m}, \quad a \geqslant 0.
\label{RBO_Case3_3_averaging_eq2}
\end{gather}

Define sets $\emptyset \neq I^x, I^y \subseteq \mathbb{N}$ as follows:
\begin{gather*}
I^x = \{ i \mid \alpha_{i, s} \neq 0 \mbox{ for some } s\,\}, \quad
I^y = \{ j \mid \alpha_{s, j} \neq 0 \mbox{ for some } s\,\}.
\end{gather*} 
For each $l \in I^y$ and $i \in I^x$, we define the sets
$J_l^x, J_i^y  \subseteq \mathbb{N}$:
$$
J_l^x = \{ s \mid \alpha_{s, l} \neq 0\},\quad
J_i^y = \{ s \mid \alpha_{i, s} \neq 0\}.
$$
Consider an arbitrary $\lambda, \mu \in \{x,y\}$ with $\lambda \neq \mu$.
Let $N_i^\lambda = |J_i^\lambda|$. If the cardinality of $J_i^\lambda$ is infinite, 
we define $N_i^\lambda = \infty$.
By definition $J^\lambda_i$, for any $i \in I^\mu$ we have $|J^\lambda_i| \geqslant 1$.
Let $k^\lambda_i$ denote the minimal element of the set $J^\lambda_i$ for all $i \in I^\mu$.
It is worth noting that $k^\lambda_i \geqslant \nu(i)$.
Finally, let $c_\lambda$ represent the minimal element of the set $I^\lambda$.

We will prove that for every $l \in I^\mu$ there exists $\Delta^\lambda_l > k^\lambda_l - \nu(l)$
such that for all $s + i, j + t \geqslant \nu(0)$,
the following formulas hold:
\begin{gather} \label{RBO_Case3_3_averaging_nonzero1}
\begin{gathered}
\alpha_{s, i} =
\begin{cases}
1, & s = k^x_i + \Delta^x_it, \ 0 \leqslant t < N^x_i, \ i \in I^y, \\
0, & \mbox{otherwise},
\end{cases} \\
\alpha_{j, t} =
\begin{cases}
1, & t = k^y_j + \Delta^y_jt, \ 0 \leqslant t < N^y_j, \ j \in I^x, \\
0, & \mbox{otherwise}.
\end{cases}
\end{gathered}
\end{gather}
Both formulas in~\eqref{RBO_Case3_3_averaging_nonzero1} are proved similarly. 
We will prove only the first one.

Fix an arbitrary $i \in I^y$.
If $N^x_i = 1$, then set $\Delta^x_i = k^x_i - \nu(i) + 1$.
If $N^x_i > 1$, then $J^x_i = \{ j_0 < j_1 < \ldots\}$.
In this case we assume $\Delta^x_i = j_1 - j_0$
and aim to prove inequality $\Delta^x_i > k^x_i - \nu(i)$.
Among all $l \geqslant \nu(i)$ the number $l = j_0$ is minimal such that $\alpha_{l, i} \neq 0$.
This implies that $\alpha_{j_0 - s, i} = 0$ for all $0 < s \leqslant j_0 - \nu(i)$.
Consequently, $\alpha_{j_0 + s, i} = 0$, since if $\alpha_{j_0 + s, i} \neq 0$,
then we arrive at a contradiction as follows:
$$
0 = 
\frac{\alpha_{j_0 - s, i}\alpha_{j_0 + s, i}}{\alpha_{j_0 + s, i}} 
\stackrel{\eqref{RBO_Case3_3_averaging}}{=}
\alpha_{2j_0 + p_x, 2i + p_y} 
\stackrel{\eqref{RBO_Case3_3_averaging_eq2}}{=}
\alpha_{j_0, i} \neq 0.
$$
Thus, we have $j_1 > 2j_0 - \nu(i)$ and conclude that $j_0 = k^x_i$, 
which implies $\Delta^x_i > k^x_i - \nu(i)$.

Let us show that for all $l \in J^x_i$,
there exists $t \geqslant 0$ such that $l = k^x_i + \Delta^x_i t$.
Let~$s$~be the smallest integer such that $j_{s + 1} - j_s \neq \Delta^x_i$.
By the definition of $\Delta_i^x$, we have $s \geq 1$.
If $j_{s + 1} - j_s < \Delta^x_i$, then a contradiction arises,
since $\alpha_{j_s + (j_{s + 1} - j_s), i}, \alpha_{j_{s}, i} \neq 0$
and $\alpha_{j_s - (j_{s + 1} - j_s), i} {=} 0$: 
$$
0 \neq
\alpha_{j_s + (j_{s + 1} - j_s), i}\alpha_{j_{s}, i} 
\stackrel{\eqref{RBO_Case3_3_averaging_eq2}}{=}
\alpha_{j_s + (j_{s + 1} - j_s), i}\alpha_{2j_{s} + p_x, 2i + p_y} 
\stackrel{\eqref{RBO_Case3_3_averaging}}{=}
\alpha_{j_s + (j_{s + 1} - j_s), i}\alpha_{j_s - (j_{s + 1} - j_s), i}
= 0.
$$
On the other hand, if $j_{s + 1} - j_s > \Delta^x_i$, 
then by~\eqref{RBO_Case3_3_averaging} we also get a contradiction,
since $\alpha_{2j_{s} + p_x, 2i + p_y}, \alpha_{j_{s - 1}, i} \neq 0$
and $\alpha_{j_s + (j_s - j_{s - 1}), i} = 0$:
$$
\alpha_{j_s - (j_s - j_{s - 1}), i}\alpha_{2j_{s} + p_x, 2i + p_y} = 
\alpha_{j_s - (j_s - j_{s - 1}), i}\alpha_{j_s + (j_s - j_{s - 1}), i}
= 0.
$$

{\sc Case 1}. 
Assume that $N^x_i = 1$ for all $i \in I^y$ or $N^y_j = 1$ for all $j \in I^x$.
We consider the first case $N^x_i = 1$.

{\sc Step 1}. 
We want to derive a formula expressing all the parameters $k^x_j$, $j \in I^y$.
From~\eqref{RBO_Case3_3_averaging_nonzero1} it follows that 
if $\alpha_{s, i} \neq 0$, then $s = k_i^x$ for any $i \in I^y$.
Therefore, from $\alpha_{(a + 1)k^x_i + ap_x, (a + 1)i + ap_y} \neq 0$ 
and by~\eqref{RBO_Case3_3_averaging_eq2}, we obtain the relation
\begin{equation} \label{RBO_Case3_3_averaging_case1_k_(a+1)i}
k^x_{(a + 1)i + ap_y} = (a + 1)k^x_i + ap_x, \quad a \geqslant 0.
\end{equation}

Consider $p_y > 0$.
Substitute $i = j$, $a = (c_y + p_y)p_y - 1$
and $i = c_y$, $a = (j + p_y)p_y - 1$ for an arbitrary $j \in I^y$
into~\eqref{RBO_Case3_3_averaging_case1_k_(a+1)i}:
$$
(c_y + p_y)p_yk^x_j + (c_y + p_y)p_yp_x - p_x = 
k^x_{p_y^3 + p_y^2(c_y + j) + p_yc_yj - p_y} =
(j + p_y)p_yk^x_{c_y} + (j + p_y)p_yp_x - p_x.
$$
Thus, parameters $k^x_j$, $j \in I^y$, can be expressed in terms of $p_x$, $p_y$, $c_y$, and $k^x_{c_y}$:
\begin{equation}\label{RBO_Case3_3_averaging_case1_k_j}
k^x_j = \frac{(j + p_y)k^x_{c_y} + (j - c_y)p_x}{c_y + p_y}, \quad j \in I^y.
\end{equation}

Now, we consider the case $p_y = 0$.
In this case we will prove that $0 \not \in I^y$.
Suppose, to the contrary, that $0 \in I^y$.
From~\eqref{RBO_Case3_3_averaging_eq2}, it follows that 
$\alpha_{k^x_0, 0} \neq 0$ and $\alpha_{2k^x_0 + p_x, 0} \neq 0$.
However, we initially assumed that $N^x_i = 1$ for all $i \in I^y$, 
which leads to a contradiction. 
Therefore, for any $i \in I^y$ it follows that $i > 0$. 
Further, by analogy to the case $p_y>0$, we put
$i = j$, $a = c_y - 1$ and $i = c_y$, $a = j - 1$
into~\eqref{RBO_Case3_3_averaging_case1_k_(a+1)i} for an arbitrary~$j \in I^y$ 
and express $k^x_j$ by $p_x$, $c_y$, and $k^x_{c_y}$.
We arrive at~\eqref{RBO_Case3_3_averaging_case1_k_j} with $p_y = 0$.
The formula~\eqref{RBO_Case3_3_averaging_case1_k_j} 
is valid for arbitrary $p_y, c_y \geqslant 0$,
since we have shown that $c_y$ and $p_y$ cannot be equal to zero simultaneously.

By substituting $\alpha_{k^x_i, i}, \alpha_{k^x_j, j} \neq 0$ 
into~\eqref{RBO_Case3_3_averaging}, we obtain
\begin{equation}\label{RBO_Case3_3_averaging_case1_k_i+j=k_i+k_j}
k^x_{i + j + p_y} = k^x_i + k^x_j + p_x, \quad i, j \in I^y.
\end{equation}
It is evident that by substituting~\eqref{RBO_Case3_3_averaging_case1_k_j}
into~\eqref{RBO_Case3_3_averaging_case1_k_i+j=k_i+k_j}, we recover the identity.

{\sc Step 2}.
We will show that $I^y = \{ c_y + \Delta s \mid s \geqslant 0\}$,
where $\Delta \geqslant \max \{c_y, 1\}$ and $\Delta \mid c_y + p_y$.

Suppose that $|(c_y, 2c_y + p_y) \cap I^y| = 0$. In this case, we assume $\Delta = c_y + p_y$.
We want to prove that the coefficients $\alpha_{i,j}$ are described by~\eqref{RBO_Case3a_main_theorem_solution}.
Assume, to the contrary, that $v = c_y + b(c_y + p_y)$, $b > 0$, 
and let $v + d$ be the minimal integer such that
$v + d \in I^y$ and $v = c_y + b(c_y + p_y)$, $0 < d < c_y + p_y$.
By the minimality of~$v+d$, we obtain that $\alpha_{s, v - d} = 0$ 
for any $s \geqslant 0$. Therefore, we have
\begin{equation} \label{RBO_Case3_3_averaging_Case1_step2_contradict_s_formula}
0 =
\alpha_{k^x_{v + d}, v + d}
\alpha_{s, v - d} \stackrel{\eqref{RBO_Case3_3_averaging}}{=}
\alpha_{k^x_{v + d}, v + d} \alpha_{k^x_{v + d} + s + p_x, 2v + p_y}.
\end{equation}
Consider $s = k^x_{2v + p_y} - k^x_{v + d} - p_x$.
By~\eqref{RBO_Case3_3_averaging_case1_k_j},
we have $s = \frac{(v + p_y - d)(k^x_{c_y} + p_x)}{c_y + p_y} - p_x$.
The inequity $s \geqslant 0$ holds, since
\begin{multline*}
(c_y + p_y)p_x \leqslant
(c_y + p_y)(k^x_{c_y} + p_x) <
(c_y + p_y + (b - 1)(c_y + p_y) + c_y + p_y - d)(k^x_{c_y} + p_x) \\ 
 = (v + p_y - d)(k^x_{c_y} + p_x).
\end{multline*}
Thus, $s$ is well-defined and
$\alpha_{k^x_{v + d} + s + p_x, 2v + p_y} = \alpha_{k^x_{2v + p_y}, 2v + p_y} \neq 0$,
so~\eqref{RBO_Case3_3_averaging_Case1_step2_contradict_s_formula} leads to a contradiction.

Consider the case $|(c_y, 2c_y + p_y) \cap I^y| \geqslant 1$.
Let $u = \min\limits_{q \in (c_y, 2c_y + p_y) \cap I^y}\{ q \}$ and $\Delta = u - c_y$.
Let us show that $\Delta \geqslant c_y$.
Due the definition of $c_y$ we have $\alpha_{s, t} = 0$ 
for all $0 \leqslant s$, $0 \leqslant t < c_y$, $0 < s + t$.
If $c_y + t \in I^y$ for some $t$, $0 \leqslant t < c_y$,
we arrive at a contradiction:
$$
0 = \alpha_{k^x_{c_y + t}, c_y + t} \alpha_{k^x_{c_y - t}, c_y - t}
\stackrel{\eqref{RBO_Case3_3_averaging}}{=}
\alpha_{k^x_{c_y + t}, c_y + t} \alpha_{2k^x_{c_y} + p_x, 2c_y + p_y} 
\stackrel{\eqref{RBO_Case3_3_averaging_case1_k_i+j=k_i+k_j}}{=}
\alpha_{k^x_{c_y + t}, c_y + t} \alpha_{k^x_{2c_y + p_y}, 2c_y + p_y}
\stackrel{\eqref{RBO_Case3_3_averaging_eq2}}{\neq}
0.
$$
Therefore, $u \geqslant 2c_y$. At the same time we have $u = \Delta + c_y$,
hence $\Delta \geqslant c_y$.
Fix $w = c_y + \Delta d$ for some $d > 0$,
and let $\xi > 0$ be the smallest integer such that $\alpha_{k^x_{w + \xi}, w + \xi} \neq 0$.
From~\eqref{RBO_Case3_3_averaging_eq2} with $\alpha_{k^x_w, w}$, we have
$\alpha_{2k^x_w + p_x, 2w + p_y} \neq 0$.
If $\xi < \Delta$, then $\alpha_{k^x_{w - \xi}, w - \xi} = 0$,
where $k^x_{w - \xi}$ is defined by~\eqref{RBO_Case3_3_averaging_case1_k_j}.
In this case, a contradiction arises:
$$
0 =
\alpha_{k^x_{w + \xi}, w + \xi} 
\alpha_{k^x_{w - \xi}, w - \xi} 
\stackrel{\eqref{RBO_Case3_3_averaging},\eqref{RBO_Case3_3_averaging_case1_k_i+j=k_i+k_j}}{=}
\alpha_{k^x_{w + \xi}, w + \xi}
\alpha_{2k^x_w + p_x, 2w + p_y} \neq 0.
$$
Similarly, if $\xi > \Delta$, then due to the definition of $\xi$
we have $\alpha_{k^x_{w + \Delta}, w + \Delta} = 0$ and
$$
0 =
\alpha_{k^x_{w - \Delta}, w - \Delta}
\alpha_{k^x_{w + \Delta}, w + \Delta} 
\stackrel{\eqref{RBO_Case3_3_averaging},\eqref{RBO_Case3_3_averaging_case1_k_i+j=k_i+k_j}}{=}
\alpha_{k^x_{w - \Delta}, w - \Delta}
\alpha_{2k^x_w + p_x, 2w + p_y} \neq 0.
$$
Thus, it must be that $\xi = \Delta$ for any minimal integer $w + \xi$. 
Consequently, for any $r \geqslant 0$, there exists $s$ 
such that $\alpha_{s, t}\neq0$, where $t = c_y + \Delta r$.

Let us apply that an arbitrary $j \in I^y$ can be expressed in the form $j = c_y + \Delta s$
and the formula~\eqref{RBO_Case3_3_averaging_case1_k_j} takes the form
$k^x_j = k^x_{c_y} + \frac{k^x_{c_y} + p_x}{c_y + p_y}\Delta s$.
In particular, by~\eqref{RBO_Case3_3_averaging} we obtain $2c_y + p_y \in I^y$,
since $c_y \in I^y$. So, $\Delta \mid c_y + p_y$
and we prove the formula~\eqref{RBO_Case3a_main_theorem_solution}.

{\sc Remark to Case~1}.
The case $N^y_j = 1$ can be considered in a similar way. 
If $p_y > 0$, then the condition $N^x_i = 1$ for all $i \in I^y$
implies that $N^y_j = 1$ for all $j \in I^x$, and vice versa.
However, if $p_y = 0$ and $N^y_j = 1$ for all $j \in I^x$, 
it may happen that $I^y = \{ 0\}$ and $N^x_0 = \infty$.

{\sc Step 2}. 
Assume that $N^x_i > 1$ for some $i \in I^y$ or $N^y_j > 1$ for some $j \in I^x$.
This is the main case in which we will prove that $T$ 
is described by one of the formulas~\eqref{RBO_Case3a_main_theorem_solution}
or~\eqref{RBO_Case3b_main_theorem_solution}.
We will show that Сases~1 and~2 cover all possible situations.
Before going into the details of Case~2, let us make one remark.

\begin{remark} \label{X_HA_X}
Let $u, v \in I^y$, $a \geqslant 0$, $b + d > 0$.
If $\alpha_{k^x_u + a, u}, \alpha_{k^x_v + b, v + d} \neq 0$,
then $\alpha_{k^x_u + a + b, u + d}\neq0$.
Indeed, if $\alpha_{k^x_u + a + b, u + d} = 0$, 
then by~\eqref{RBO_Case3_3_averaging} we have a contradiction:
\begin{gather*}
0 \neq 
\alpha_{k^x_u + a, u} \alpha_{k^x_v + b, v + d} =
\alpha_{k^x_u + a, u} \alpha_{k^x_u + k^x_v + a + b + p_x, u + v + d + p_y}, \\
0 =
\alpha_{k_v, v} \alpha_{k^x_u + a + b, u + d} =
\alpha_{k_v, v} \alpha_{k^x_u + k^x_v + a + b + p_x, u + v + d + p_y}.
\end{gather*}
\end{remark}

{\sc Step 1}.
Assume that there exists $i \in I^y$ such that $N^x_i > 1$.
The case when $N^x_j > 1$ for some $j \in I^x$ can be considered  analoguesly.
Let us prove the following identity by induction:
\begin{equation}\label{RBO_Case3_3_averaging_Delta_specification_I}
\Delta^x_{(a + 1)i + ap_y} = \Delta^x_i, \quad
a \geqslant 0.
\end{equation}

It is obvious that~\eqref{RBO_Case3_3_averaging_Delta_specification_I} holds for $a = 0$.
Now we assume that the formula is true for all $s < a$.
From~\eqref{RBO_Case3_3_averaging_eq2} we have $\alpha_{bk^x_i + (b - 1)p_x, bi + (b - 1)p_y} \neq 0$, $b > 0$,
and by~\eqref{RBO_Case3_3_averaging_nonzero1}, it follows that $\alpha_{k^x_i + \Delta^x_i, i} \neq 0$.
Substituting these coefficients into~\eqref{RBO_Case3_3_averaging},
we obtain $\alpha_{(b + 1)k^x_i + \Delta^x_i + bp_x, (b + 1)i + bp_y} \neq 0$ for all $b > 0$.
Moreover, $(a + 1)i + ap_y \in I^y$.
Therefore, by~\eqref{RBO_Case3_3_averaging_nonzero1}, 
there exist integers $t_2 > t_1 \geqslant 0$ such that
\begin{gather*}
k^x_{(a + 1)i + ap_y} + \Delta^x_{(a + 1)i + ap_y}t_1 = (a + 1)k^x_i + ap_x, \\
k^x_{(a + 1)i + ap_y} + \Delta^x_{(a + 1)i + ap_y}t_2 = (a + 1)k^x_i + \Delta^x_i + ap_x.
\end{gather*}
This implies that $\Delta^x_{(a + 1)i + ap_y} \mid \Delta^x_i$.
Suppose that $\Delta^x_{(a + 1)i + ap_y} < \Delta^x_i$,
then by~\eqref{RBO_Case3_3_averaging_nonzero1} we find
$\alpha_{(a + 1)k^x_i + \Delta^x_{(a + 1)i + ap_y} + ap_x, (a + 1)i + ap_y} \neq 0$ and
\begin{gather*}
\alpha_{k^x_i, i}
\alpha_{ak^x_i + \Delta^x_{(a + 1)i + ap_y} + (a - 1)p_x, ai + (a - 1)p_y}
\stackrel{\eqref{RBO_Case3_3_averaging}}{=}
\alpha_{k^x_i, i}
\alpha_{(a + 1)k^x_i + \Delta^x_{(a + 1)i + ap_y} + ap_x, (a + 1)i + ap_y} \neq 0.
\end{gather*}
However, again by~\eqref{RBO_Case3_3_averaging_nonzero1}, we have
$\alpha_{ak^x_i + \Delta^x_{(a + 1)i + ap_y} + (a - 1)p_x, ai + (a - 1)p_y} = 0$,
since $\Delta^x_{(a + 1)i + ap_y} < \Delta^x_i = \Delta^x_{ai + (a - 1)p_y}$.
This contradiction completes the inductive step 
and proves~\eqref{RBO_Case3_3_averaging_Delta_specification_I}.

{\sc Step 2}.
We now prove that $\Delta_j^x = \Delta_i^x$ for all $i, j \in I^y$, $N^x_i, N^x_j > 1$.
From~\eqref{RBO_Case3_3_averaging_eq2}, it follows that
$(d + 1)i + dp_y, (b + 1)j + bp_y \in I^y$ for all $b, d \geqslant 0$.

If $p_y > 0$, then we choose $d = (j + p_y)p_y - 1$ and $b = (i + p_y)p_y - 1$
and by~\eqref{RBO_Case3_3_averaging_Delta_specification_I} obtain
$$
\Delta^x_j =
\Delta^x_{(i + p_y)p_yj + ((i + p_y)p_y - 1)p_y} =
\Delta^x_{p_y^3 + p_y^2(i + j) + p_y(ij - 1)} =
\Delta^x_{(j + p_y)p_yi + ((j + p_y)p_y - 1)p_y} =
\Delta^x_i.
$$

If $p_y = 0$, then $d = j - 1$, $b = i - 1$, where $i, j > 0$. 
Once again, applying~\eqref{RBO_Case3_3_averaging_Delta_specification_I},
we obtain $\Delta^x_j = \Delta^x_i$.
Study the subcase $i = 0$ and $j > 0$.
From~\eqref{RBO_Case3_3_averaging_eq2} considered for $\alpha_{k^x_0, 0}$ we get $N^x_0 = \infty$.
Thus, by~\eqref{RBO_Case3_3_averaging_nonzero1},
it follows that $\alpha_{k^x_0 + \Delta^x_0s, 0} \neq 0$ for all $s \geqslant 0$.
Let us apply~\eqref{RBO_Case3_3_averaging}:
$$
\alpha_{k^x_j, j} \alpha_{k^x_j + k^x_0 + \Delta^x_0s + p_x, j} =
\alpha_{k^x_j, j} \alpha_{k^x_0 + \Delta^x_0s, 0} 
\neq 0.
$$
We observe that $\alpha_{k^x_j + k^x_0 + \Delta^x_0s + p_x, j} \neq 0$,
and therefore, by~\eqref{RBO_Case3_3_averaging_nonzero1},
for each $s \geqslant 0$, there exists $t \geqslant 0$ such that
$k^x_j + k^x_0 + \Delta^x_0s + p_x = k^x_j + \Delta^x_j t$.
Substituting $s = 0$, we deduce $\Delta^x_j \mid k^x_0 + p_x$,
and substituting $s = 1$, we obtain $\Delta^x_j \mid \Delta^x_0$.
Suppose to the contrary that $\Delta^x_j < \Delta^x_0$. 
Then we have $\alpha_{k^x_0 + \Delta^x_j, 0}~=~0$ by~\eqref{RBO_Case3_3_averaging_nonzero1}.
Taking  $(u, v, a, b, d) = (0, j, 0, \Delta^x_j, 0)$,
we arrive at a contradiction due to Remark~\ref{X_HA_X}.
Thus, we have shown that $\Delta^x_j = \Delta^x_i$ 
for all $i, j \in I^y$, $N^x_i, N^x_j > 1$, when~$p_y = 0$.
Analogously we can show that 
$\Delta_j^y = \Delta_i^y$ 
for all $i, j \in I^x$, $N^y_i, N^y_j > 1$.

{\sc Step 3}.
In this step, we prove that for any $\lambda \in \{ x, y \}$,
if $N^\lambda_s > 1$ for some $s \in I^\lambda$,
then necessarily $N^\lambda_j > 1$ for all $j \in I^\lambda$.
Consider $\lambda = y$, the case $\lambda = x$ follows by symmetry.

Assume to the contrary that there exists $i \in I^y$ such that $N^x_i = 1$.
Then, by~\eqref{RBO_Case3_3_averaging_nonzero1}, we have
$\alpha_{s, i} \neq 0$ if and only if $s = k^x_i$.
Let $j \in I^y$ be such that $N^x_j > 1$.
Then, again by~\eqref{RBO_Case3_3_averaging_nonzero1}, 
we get $\alpha_{k^x_j, j}, \alpha_{k^x_j + \Delta^x_j, j} \neq 0$.
We get a contradiction with Remark~\ref{X_HA_X} for $(u, v, a, b, d) = (i, j, 0, \Delta^x_j, 0)$.
So, we prove that $N^x_i > 1$ for all~$i \in I^y$.

{\sc Step 4}.
In this step, we show that precisely one of the following three mutually exclusive cases holds:
\begin{enumerate}
\item[(i)] $I^y \neq \{ 0\}$, $N^\lambda_i = 1$ for all $i \in I^\lambda$ and $\lambda\in \{x,y\}$,
\item[(ii)] $I^y = \{ 0\}$, $N^x_0 = \infty$ and $N^y_i = 1$ for all $i \in I^x$,
\item[(iii)] $I^y \neq \{ 0\}$, $N^\lambda_i > 1$ for all $i \in I^\lambda$ and $\lambda\in \{x,y\}$.
\end{enumerate}
Note that the variant $(\mathrm{ii})$ can only arise when $p_y = 0$, 
and it has already been considered in Case~1. 
After completing this step,
we will obtain that variants $(\mathrm{i})$ and $(\mathrm{ii})$ are devoted to Case~1,
while the variant $(\mathrm{iii})$ will be considered in Case~2 in further steps.

Let us show that 
if $p_y = 0$, then we are in Case $(\mathrm{ii})$; 
if $p_y > 0$, then the condition $N^x_i > 1$ for all $i \in I^y$
implies $N^y_j > 1$ for all $j \in I^x$.
Due to Step~3 it is enough to show that there exists $l \in I^x$ such that $N^y_l > 1$.

Consider an arbitrary $i \in I^y$.
By our assumption and due to~\eqref{RBO_Case3_3_averaging_nonzero1},
we have $\alpha_{k^x_i, i}, \alpha_{k^x_i + \Delta^x_i, i} \neq 0$.
By~\eqref{RBO_Case3_3_averaging_eq2} we obtain
\begin{gather} 
\begin{gathered} \label{RBO_Case3_3_averaging_intermediate_calculations_step4}
\alpha_{
(k^x_i + \Delta^x_i + p_x)k^x_i + (k^x_i + \Delta^x_i + p_x - 1)p_x, 
(k^x_i + \Delta^x_i + p_x)i + (k^x_i + \Delta^x_i + p_x - 1)p_y
} \neq 0, \\
\alpha_{(
k^x_i + p_x)(k^x_i + \Delta^x_i) + (k^x_i + p_x - 1)p_x, 
(k^x_i + p_x)i + (k^x_i + p_x - 1)p_y
} \neq 0.
\end{gathered}
\end{gather}
Let us denote the first component of both coefficients by
$j = (k^x_i)^2 + k^x_i(2p_x + \Delta^x_i) + (p_x + \Delta^x_i-1)p_x$ and note that $j \in I^y$.
Moreover, the difference between the second components of the coefficients equals $A = \Delta^x_i (i + p_y)$.
If $p_y > 0$, then $A > 0$.
Suppose that $p_y = 0$, then $A = 0$ if and only if $i = 0$.
Since $i$ was chosen arbitrarily, it follows that $I^y = \{0\}$.
This corresponds precisely to Case $(\mathrm{ii})$.
If $A > 0$, then $N^y_j > 1$ and we have proved the required by Step~3.

Now we show that if $N^y_i > 1$ for all $i \in I^x$,
then $N^x_j > 1$ for all $j \in I^y$.
Up to a relabeling of indices, for $p_y > 0$, 
the structure of the proof remains unchanged, since $p_x > 0$ holds by assumption.
In the case $p_y = 0$ 
for the analogue 
of~\eqref{RBO_Case3_3_averaging_intermediate_calculations_step4}
we obtain the expression $B_i = k^y_i + p_y - 1$ which may be negative.

Consider arbitrary $i, j \in I^x$.
We can assume $i \neq j$, since $p_x > 0$ and by~\eqref{RBO_Case3_3_averaging_eq2}
we have $2i + p_x \in I^x$, i.\,e. $|I^x| \neq 1$.
The main problem to repeat the previous arguments is when $B_i < 0$ for all $i \in I^x$.
Note that $B_i < 0$ if and only if $k^y_i = 0$.
Thus, to complete the proof of the required statement, 
it remains to consider the case when $k^y_i = 0$ for all $i \in I^x$.
According to~\eqref{RBO_Case3_3_averaging_nonzero1}, we have 
$\alpha_{i, \Delta^y_i}, \alpha_{j, \Delta^y_j} \neq 0$.
Moreover, we know from Step~2 that $\Delta^y_i = \Delta^y_j$,
and hence $N^x_{\Delta^y_i} > 1$. 
By Step~3, and by the fact that $N^x_l > 1$ holds for all $l \in I^y$, 
the proof is completed. 

{\sc Intermediate result}.
We prove in Case~2 that only two variants are possible: $(\mathrm{ii})$ and $(\mathrm{iii})$. 
Further, the variant~$(\mathrm{ii})$ was studied within Case~1.  
Therefore, to complete the analysis of Case~2, 
it remains to consider the variant~$(\mathrm{iii})$.

Recall that in the variant $(\mathrm{iii})$ we have $|I^x|, |I^y| > 1$ 
and $N^\lambda_i > 1$ for all 
$\lambda, \mu \in \{ x, y\}$, $\lambda \neq \mu$ and $i \in I^\mu$.
Moreover, there exist constants $\Delta_\lambda > 0$ such that $\Delta^\lambda_i = \Delta_\lambda$ 
and $\Delta_\lambda > k^\lambda_i - \nu(0)$ for all $i \in I^\lambda$.
The formulas~\eqref{RBO_Case3_3_averaging_nonzero1} take the following form:
\begin{gather} \label{RBO_Case3_3_averaging_case2_last_formula1}
\begin{gathered} 
\alpha_{s, i} =
\begin{cases}
1, & s = k^x_i + \Delta_x t, \ 0 \leqslant t < N^x_i, \ i \in I^y, \\
0, & \mbox{otherwise};
\end{cases} \\
\alpha_{j, s} =
\begin{cases}
1, & s = k^y_j + \Delta_y t, \ 0 \leqslant t < N^y_j, \ j \in I^x, \\
0, & \mbox{otherwise}.
\end{cases}
\end{gathered}
\end{gather}

{\sc Step 5}.
We will prove the following formulas:
\begin{equation} \label{RBO_Case3_3_averaging_I_J_formula}
I^x = \{ c_x + \Delta_x s \mid s \geqslant 0\}, \quad
I^y = \{ c_y + \Delta_y s \mid s \geqslant 0\}.
\end{equation}

Let $d > 0$ be a minimal integer such that $c_y + \Delta_y d \notin I^y$.
Hence, $\alpha_{s, c_y + \Delta_y d} = 0$ for all $s \geqslant 0$.
Show that the case $d = 1$ leads to a contradiction with 
the assumption that $N^y_j > 1$ for all $j \in I^x$.
We have $k^x_{c_y} \in I^x$, since $\alpha_{k^x_{c_y}, c_y} \neq 0$.
By the minimality of~$c_y$, we conclude $c_y = k^y_{k^x_{c_y}}$.
However, we also have
$\alpha_{k^x_{c_y}, c_y + \Delta_y} = \alpha_{k^x_{c_y}, k^y_{k^x_{c_y}} + \Delta_y} = 0$
and $N^y_{c_y + \Delta_y} = 1$.
If $d > 1$, then we obtain a contradiction with~\ref{X_HA_X} for $w = c_y + \Delta_y(d - 1)$ considering $(u, v, a, b, d) = (w, c_y, 0, 0, \Delta_y)$.

Dealing analogously with $I^x$,
we conclude that the formulas~\eqref{RBO_Case3_3_averaging_I_J_formula} hold.
By~\eqref{RBO_Case3_3_averaging_eq2}, 
we obtain
$\alpha_{2k^x_{c_y} + p_x, 2c_y + p_y} \neq 0$ 
($\alpha_{2c_x + p_x, 2k^y_{c_x} + p_y} \neq 0$)
for any $i \in I^y$ ($j \in I^x$).
From this joint with~\eqref{RBO_Case3_3_averaging_I_J_formula},
it follows that $\Delta_y \mid c_y + p_y$ ($\Delta_x \mid c_x + p_x$).
In other words, for each $\lambda \in \{x, y\}$ 
there exist $r_\lambda > 0$ such that 
$r_\lambda\Delta_\lambda = c_\lambda + p_\lambda$.

{\sc Step 6}.
In this step, we will demonstrate that the assumptions of the variant $(\mathrm{iii})$
imply $N^\lambda_i = \infty$ for all $i \in I^\lambda$ and $\lambda\in \{x,y\}$.

First of all we prove that if $N^\lambda_i = \infty$ for some $i \in I^{\mu}$,
then $N^\lambda_s~=~\infty$ for all $s \in I^{\mu}$.
We will present the argument for the case $I^y$,
the case $I^x$ can be considered analogously.
Let us assume to the contrary that $N^x_s < \infty$ 
and $N^x_i = \infty$ for some $s, i \in I^y$.
Then we obtain $\alpha_{k^x_s + \Delta_x N^x_s, s} = 0$ by~\eqref{RBO_Case3_3_averaging_case2_last_formula1}.
This leads to a contradiction with Remark~\ref{X_HA_X}
considered with parameters $(u, v, a, b, d) = (s, i, \Delta_x(N^x_s - 1), \Delta_x, 0)$.
Therefore, it must be that either $N^x_i = \infty$ or $N^x_i < \infty$ for all $i \in I^y$.

Let $\lambda, \mu \in \{x, y\}$ and $\lambda \neq \mu$.
While deriving formula~\eqref{RBO_Case3_3_averaging_nonzero1}
we have proved that $k^\lambda_i - \nu(i) < \Delta^\lambda_i = \Delta_\lambda$.
Hence, all $k^\lambda_i$, $i \in I^\mu$, are uniformly bounded, a contradiction with $|I^\mu| = \infty$.
Hence, there exists $l < \Delta_\lambda$ such that $N^\mu_l = \infty$.
By the previous paragraph we have $N^\lambda_i = \infty$ for all $i \in I^\lambda$.

{\sc Step 7}.
At this step, we will find an important relation fulfilled for $k^\lambda_i$.
Substituting nonzero coefficients into~\eqref{RBO_Case3_3_averaging},
as well as using~\eqref{RBO_Case3_3_averaging_case2_last_formula1}
for all $\lambda, \mu \in \{x, y\}$, $\lambda \neq \mu$,
we derive that for some $\sigma^\lambda_{i, j} \geqslant 0$ the following relation holds:
\begin{equation*}
k^\lambda_{i + j + p_{\mu}} = 
k^\lambda_i + k^\lambda_j + p_\lambda - \Delta_\lambda \sigma^\lambda_{i, j},
\ i, j \in I^\mu, \ \sigma^\lambda_{i, j} \geqslant 0.
\end{equation*}
Let us apply~\eqref{RBO_Case3_3_averaging_I_J_formula} and the equality
$r_\lambda \Delta_\lambda = c_\lambda + p_\lambda$,
where $r_\lambda = 0$ if $c_\lambda = p_\lambda = 0$.
To simplify notation, we define $\widetilde{k}^\lambda_s = k^\lambda_{c_{\mu} + \Delta_{\mu} s}$,
$
\xi^\lambda_{s, t} = 
\sigma^\lambda_{c_{\mu} + \Delta_{\mu} s, c_{\mu} + \Delta_{\mu} t}
$
for all $s, t \geqslant 0$,
then the last relation can be rewritten as 
\begin{equation} \label{RBO_Case3_3_averaging_k_i+j+q}
\widetilde{k}^\lambda_{s + t + r_{\mu}} = 
\widetilde{k}^\lambda_s + \widetilde{k}^\lambda_t + p_\lambda - \Delta_\lambda \xi^\lambda_{s, t},
\quad s, t \geqslant 0.
\end{equation}
The solution to the recurrence relation~\eqref{RBO_Case3_3_averaging_k_i+j+q} with $r_{\mu} = 0$ 
was considered in Lemma~\ref{k_i_recurrent_lemma_1}
and with $r_{\mu} > 0$ in Lemma~\ref{k_i_recurrent_lemma_2}.

{\sc Step 8}.
We know by Lemmas~\ref{k_i_recurrent_lemma_1} and~\ref{k_i_recurrent_lemma_2}
how to express $\widetilde{k}^\lambda_i$ via 
$\widetilde{k}^\lambda_0$,
$\{\xi_{0,i}^\lambda \}^{\infty}_{i = 0} \cup \{\xi_{1,i}^\lambda \}^{r - 1}_{i = 1}$.
It remains to verify that substituting the formulas coming from these two Lemmas into~\eqref{RBO_Case3_3_averaging_case2_last_formula1},
the relation \eqref{RBO_Case3_3_averaging} holds.

All coefficients~$\alpha_{s,i}$ are uniquely described
by one of the formulas~\eqref{RBO_Case3_3_averaging_case2_last_formula1}.
Let us rewrite the first formula from~\eqref{RBO_Case3_3_averaging_case2_last_formula1}
taking into account~\eqref{RBO_Case3_3_averaging_I_J_formula}:
\begin{equation} \label{RBO_Case3_3_averaging_last_solution}
\alpha_{s, t} =
\begin{cases}
1, & 
s = k^x_{c_y + \Delta_y j} + \Delta_x i, \
t = c_y + \Delta_y j, \ i, j \geqslant 0, \\
0, & \mbox{otherwise},
\end{cases}
\end{equation}
We will verify that this relation holds in the case $r_y > 0$, 
the verification for the case $r_y = 0$ is analogous.

By the symmetry of the relation~\eqref{RBO_Case3_3_averaging},
it sufficient to consider only the following two cases: 
$\alpha_{n, m}, \alpha_{s, t} \neq 0$ and $\alpha_{n, m} = 0$, $\alpha_{s, t} \neq 0$.
Coefficients $\alpha_{k^x_{c_y + \Delta_y a} + \Delta_x b, c_y + \Delta_y a}$,
$\alpha_{k^x_{c_y + \Delta_y u} + \Delta_x v, c_y + \Delta_y u}$
with $0 \leqslant a, b, u, v$ are nonzero by~\eqref{RBO_Case3_3_averaging_last_solution}.
Substituting these values into~\eqref{RBO_Case3_3_averaging}, we obtain
\begin{multline*}
\alpha_{k^x_{c_y + \Delta_y a} + \Delta_x b, c_y + \Delta_y a}
\alpha_{k^x_{c_y + \Delta_y u} + \Delta_x v, c_y + \Delta_y u} \\ = 
\alpha_{k^x_{c_y + \Delta_y a} + \Delta_x b, c_y + \Delta_y a}
\alpha_{
k^x_{c_y + \Delta_y a} + \Delta_x b + k^x_{c_y + \Delta_y u} + \Delta_x v + p_x, 
c_y + \Delta_y a + c_y + \Delta_y u + p_y
}.
\end{multline*}
Hence, 
$
\alpha_{
k^x_{c_y + \Delta_y a} + k^x_{c_y + \Delta_y u} + \Delta_x (b + v) + p_x, 
2c_y + \Delta_y(a + u) + p_y} \neq0
$.
Since $c_y + p_y = r_y\Delta_y$, we have $2c_y + \Delta_y(a + u) + p_y = c_y + \Delta_y(a + u + r_y)$
and by~\eqref{RBO_Case3_3_averaging_last_solution},
there exists $l \geqslant 0$ such that the relation
$$
k^x_{c_y + \Delta_y a} + k^x_{c_y + \Delta_y u} + \Delta_x (b + v) + p_x 
 = k^x_{c_y + \Delta_y(a + u + r_y)} + \Delta_x l
$$
holds. From the following equalities we obtain that $l = b + v + \xi^x_{a, u}$:
\begin{multline*}
\Delta_x l = 
k^x_{c_y + \Delta_y a} + k^x_{c_y + \Delta_y u} + \Delta_x (b + v) + p_x - k^x_{c_y + \Delta_y(a + u + r_y)} =
\widetilde{k}^x_a + \widetilde{k}^x_u - \widetilde{k}^x_{a + u + r_y} + \Delta_x (b + v) + p_x
\\ \stackrel{\eqref{RBO_Case3_3_averaging_k_i+j+q}}{=}
\Delta_x \xi^x_{a, u} - p_x + \Delta_x (b + v) + p_x = 
\Delta_x(b + v + \xi^x_{a, u}).
\end{multline*}

Let us consider the case $\alpha_{k^x_{c_y + \Delta_y a} + \Delta_x b, c_y + \Delta_y a} \neq 0$
for $0 \leqslant a, b$ and $\alpha_{s, t} = 0$,
where there are no $0 \leqslant u, v$ such that both conditions
$s = k^x_{c_y + \Delta_y u} + \Delta_x v$ and $t = c_y + \Delta_y u$ are satisfied
(see~\eqref{RBO_Case3_3_averaging_last_solution}).
Substituting these coefficients into~\eqref{RBO_Case3_3_averaging},
we obtain
\begin{equation} \label{FinalCheckZeroNonzero}
\alpha_{k^x_{c_y + \Delta_y a} + \Delta_x b + s + p_x, c_y + \Delta_y a + t + p_y} = 0.
\end{equation}

If $t \not \in c_y + \Delta_y \mathbb{N}$,
then~\eqref{FinalCheckZeroNonzero} holds by~\eqref{RBO_Case3_3_averaging_last_solution}.
Suppose that $t = c_y + \Delta_y u$ for some $u \geqslant 0$, 
then $s \not \in k^x_{c_y + \Delta_y u} + \Delta_x \mathbb{N}$. 
In order for~\eqref{FinalCheckZeroNonzero} to hold, it must be that
$k^x_{c_y + \Delta_y a} + \Delta_x b + s + p_x \not \in k^x_{c_y + \Delta_y(a + u + r_y)} + \Delta_x \mathbb{N}$.
Let us apply~\eqref{RBO_Case3_3_averaging_k_i+j+q}:
$$
k^x_{c_y + \Delta_y a} - k^x_{c_y + \Delta_y(a + u + r_y)} = 
\widetilde{k}^x_{a} - \widetilde{k}^x_{a + u + r_y} =
-\widetilde{k}^x_{u} - p_x + \Delta_x \xi^x_{a, u},
$$
it follows that
$$
k^x_{c_y + \Delta_y a} + \Delta_x b + s + p_x- k^x_{c_y + \Delta_y(a + u + r_y)} =
-\widetilde{k}^x_{u} + \Delta_x (\xi^x_{a, u} + b) + s \neq 0 \mod \Delta_x,
$$
since
$s -  k^x_{c_y + \Delta_y u} = s - \widetilde{k}^x_{u}\neq 0 \mod \Delta_x$.
So, we prove the formula~\eqref{RBO_Case3b_main_theorem_solution}.

\newpage

We have completed the analysis of Case~2 and have shown that any averaging operator
defined as $T(x^n y^m) = \alpha_{n, m} x^{n + p_x} y^{m + p_y}$
is expressed only by one of the three formulas~\eqref{RBO_Case3a_main_theorem_solution},~\eqref{RBO_Case3b_main_theorem_solution},
or~\eqref{RBO_Case3b_main_theorem_solution_p_y=0}.
\end{proof}

\subsection{Case $p + q > 0$ for Rota---Baxter operators}

\begin{corollary} \label{RBO_Case3_corollary}
Let $R$ be an RB-operator of weight zero on $F[x,y]$ or $F_0[x,y]$
defined by the rule $R(x^n y^m) = \gamma_{n, m} T(x^n y^m)$, 
where $p_x > 0$, $p_y \geqslant 0$, 
and $T$ is an operator from Theorem~\ref{RBO_Case3_main_theorem}.
Then $R$ has one of the following forms.

In Case~a), the following formula holds for all $n + m \geqslant \nu(0)$:
\begin{equation} \label{RBO_Case3_3_RBO_solution1}
\gamma_{n, m} =
\begin{cases}
\dfrac{r\gamma_{k, c}}{r + s}, & 
n = k + \frac{k + p_x}{c + p_y}\Delta s, \ m = c + \Delta s, \  s \geqslant 0,\\
0, & \mbox{otherwise}.
\end{cases}
\end{equation}

In Case b), when $p_y = 0$, 
we have $\Delta_x \sigma_0 = \Delta_x r = k_0 + p_x$ and
\begin{equation} \label{RBO_Case3_3_RBO_solution2}
\gamma_{s, t} =
\begin{cases}
\dfrac{r\gamma_{k_0, 0}\gamma_{k_1, c}}
{rv\gamma_{k_0, 0} + (u + r - \sum_{s = 0}^{v - 1}\sigma_s)\gamma_{k_1, c}}, 
& s = k_v + \Delta u, \ t = cv, \ u, v\geqslant 0, \\
0. & \mbox{otherwise},
\end{cases}
\end{equation}
for all $s + t \geqslant \nu(0)$.
In Case b), when $0 < p_y$, one has
\begin{gather}
\gamma_{s, t} =
\begin{cases}
\dfrac{r_y}
{\frac{r_yu - H(v)}{\gamma_{k_0 + \Delta_x, c_y}} + \frac{H(v) + v - r_y(u - 1)}{\gamma_{k_0, c_y}}}, 
& s = k_v + \Delta_x u, \ t = c_y + \Delta_yv, \ u, v\geqslant 0, \\
0, & \mbox{otherwise},
\end{cases} \label{RBO_Case3_3_RBO_solution3}\\
H(v:= ar_y + b)
 = r_y\sum_{s = 0}^{a - 1}\sigma_{0, sr_y + b} 
 - r_y\sum_{s = 1}^{b - 1} (\sigma_{0, s + 1} {-} \sigma_{1, s}) 
 + b\sigma_{0, 0} + b\sum_{s = 1}^{r_y - 1} (\sigma_{0, s + 1} {-} \sigma_{1, s}), \label{RBO_Case3_3_RBO_solution3_H(v)}
\end{gather}
for all $0 \leqslant a$ and $0 \leqslant b < r_y$.
\end{corollary}

\begin{remark}
The expression for $H(v)$~\eqref{RBO_Case3_3_RBO_solution3_H(v)}
coincides with the expression in last parentheses in~\eqref{k_i_recurrent_sequence2_last}.
\end{remark}

\begin{remark}
Theorem~\ref{RBO_Case3_main_theorem} holds over any field $F$ 
regardless of characteristic, whereas Lemma~\ref{RBO_Case3_corollary} 
holds only when characteristic of~$F$ is zero.
\end{remark}

\begin{example}
Let us consider RB-operator on $F[x, y]$ from Corollary~\ref{RBO_Case3_corollary} with $p_y = 0$.
Fix some $d > 0$ and assume that 
$c = \Delta$, $\Delta r = p_x$.
Choosing $k_0 = 0$, $k_1 = dp_x$, $\sigma_s = dr$ for all $s > 0$,
we obtain $k_n = (n + d - 1)p_x$, $n > 0$, and~\eqref{RBO_Case3_3_RBO_solution2} takes the following form:
$$
\gamma_{s, t} =
\begin{cases}
\dfrac{r\gamma_{0, 0}\gamma_{dp_x, c}}
{rv\gamma_{0, 0} + (u + r(1 - vd))\gamma_{dp_x, c}}, 
& s = (v + d - 1)p_x + c u, \ t = cv, \ u, v\geqslant 0, \\
0, & \mbox{otherwise}.
\end{cases}
$$
\end{example}

\section{Discussion}

One of the most notable examples among considered monomial RB-operators is the case when
$\alpha_{n, m} \neq 0$ for all $n + m \geqslant \nu(0)$.
Below, we describe all such RB-operators coming from averaging operators. 
It is important to note that not all possible parameter configurations 
admit an RB-operator whose kernel does not contain monomials.

In Theorems~\ref{Case2_theorem} and~\ref{Case1_theorem},  
all coefficients $\alpha_{n,m}$ are nonzero only if $r + c > 0$.
In Theorem~\ref{Case2_theorem} we have $I = \mathbb{N}$, $k_i = \nu(i)$, $\Delta = 1$.
Hence,~\eqref{RBO_Case2b_solution} takes the form
$\alpha_{n, s} = (rn + c + \nu(n)) \alpha_{n, 0} / (rn + c + s)$ for all $n + s \geqslant \nu(0)$.
In Theorem~\ref{Case1_theorem} we have 
$I = \mathbb{Z}$, $\Delta = 1$, and $k_i = \zeta_r(i)$ for all $i \in \mathbb{Z}$.
So, we rewrite~\eqref{RBO_Case1b_solution} for $(r + 1)l + t \geqslant \nu(0)$ as
$$
\alpha_{rl + t, l} =
\dfrac{(\zeta_r(t) + c)\alpha_{r\zeta_r(t) + t, \zeta_r(t)}}{\zeta_r(t) + c + s}, 
\ t\in \mathbb{Z}, \ l = \zeta_r(t) + s, \ s \geqslant 0.
$$

In Theorem~\ref{RBO_Case3_p=0_q=0_theorem}, the formula~\eqref{RBO_Case3_1_F0_solution_second} 
takes the form
$$
\alpha_{s, t} = \frac{\alpha_{1, 0}\alpha_{0, 1}}{t\alpha_{1, 0} + s\alpha_{0, 1}}, 
\quad s + t > 0.
$$

Finally, we consider Corollary~\ref{RBO_Case3_corollary}.
When $p_y = 0$ we have $c = \Delta = 1$, $\sigma_0 = r = p_x + \nu(0)$, 
$k_n = \nu(n)$, $n \geqslant 0$, and $\sigma_n = p_x$, $n > 0$. 
Thus, the formula~\eqref{RBO_Case3_3_RBO_solution2}
takes the form
$$
\gamma_{s, t} = 
\dfrac{(p_x + \nu(0))\gamma_{\nu(0), 0}\gamma_{0, 1}}
{t(p_x + \nu(0))\gamma_{\nu(0), 0} + (s - (t - 1)p_x)\gamma_{0, 1}}, 
\ r, s \geqslant 0, \ r + t \geqslant \nu(0).
$$
When $p_y > 0$ we arrive at
$\Delta_x = \Delta_y = 1$, $c_x = c_y = 0$, $r_x = p_x$, $r_y = p_y$,
$k_n = \nu(n)$, $n \geqslant 0$.
Also, $\sigma_{i, j} = p_x + \nu(i) + \nu(j)$, where $i = 0$ and $0 \leqslant j$
or $i = 1$ and $1 \leqslant j \leqslant p_y - 1$.
The formula~\eqref{RBO_Case3_3_RBO_solution3_H(v)} takes the form 
$H(v) = vp_x + \nu(0)(p_x + p_y)$, $v \geqslant 0$,
and~\eqref{RBO_Case3_3_RBO_solution3} for all $s + t \geqslant \nu(0)$ can be rewritten as follows,
$$
\gamma_{s, t} = 
\dfrac{p_y}
{
\dfrac{p_ys - tp_x}{\gamma_{\nu(0) + 1, 0}} + 
\dfrac{t(p_x + 1) - p_y(s - 1)}{\gamma_{\nu(0), 0}} +
\nu(0)(p_x + p_y)\left(\dfrac{1}{\gamma_{\nu(0), 0}} - \dfrac{1}{\gamma_{\nu(0) + 1, 0}}\right)
}.
$$

\section*{Acknowledgements}

Author is grateful to his supervisor Vsevolod Gubarev for very valuable discussions
and to Maxim Goncharov for several important suggestions.

The study was supported by a grant from the Russian Science Foundation \linebreak
\textnumero 23-71-10005, https://rscf.ru/project/23-71-10005/

\noindent Artem Khodzitskii \\
Novosibirsk State University, \\
Pirogova str. 2, 630090 Novosibirsk, Russia \\
e-mail: a.khodzitskii@mail.ru

\newpage
\section*{Appendix A}

Let $P$ be a monomial operator defined by the rule  
$P(u) = \varepsilon_u v$, $u, v \in M(X)$, 
such that $P(a)P(b), P(P(a)b), P(aP(b))$ are proportional for any monomials $a, b$.
For~$P$, the relation~\eqref{RBO} is equivalent to
$$
\varepsilon_a\varepsilon_b =
\varepsilon_a\varepsilon_{P(a)b} + \varepsilon_b \varepsilon_{aP(b)}.
$$
If $\varepsilon_a = \varepsilon_b = 0$, the equation reduces to the trivial identity $0 = 0$.
In what follows, we will omit this trivial case.

\begin{lemma}[Lemma \ref{main_lemma_req_seq}]
Fix $d, \tau \geqslant 0$ and suppose that there exists a sequence $\{\beta_i\}_{i = \tau}^\infty \in F$
satisfying the recurrence relation
\begin{equation} \label{main_lemma_req_seq_eq}
\beta_s \beta_t = (\beta_s + \beta_t)\beta_{s + t + d}, \quad s, t \geqslant \tau.
\end{equation}
Then one of the following holds: 
either $\beta_i = 0$ for all $i \geqslant \tau$ 
or there exist integers $k$~and~$\Delta$ such that $\beta_k \neq 0$, with
$0 \leqslant k - \tau < \Delta \leqslant k + d$, 
$\Delta \mid k + d$, such that for any $t \geqslant \tau$
one has
\begin{equation} \label{main_lemma_req_seq_solution}
\beta_t = 
\begin{cases}
\dfrac{(k + d)\beta_k}{k + d + \Delta s}, & t = k + \Delta s, \ s \geqslant 0, \\
0, & \mbox{otherwise}.
\end{cases}
\end{equation}
In particular, if $d = \tau = 0$, then $\beta_i = 0$ for all $i \geqslant 0$.
\end{lemma}

\begin{proof}
Suppose that $d = \tau = 0$. Substituting $t = 0$ into~\eqref{main_lemma_req_seq_eq}, 
we obtain $\beta_s \beta_0 = \beta_s^2 + \beta_s\beta_0$,
which implies that $\beta_s^2 = 0$ for all $s \geqslant 0$.

Assume that $d + \tau > 0$ and $\beta_i \neq 0$ for some $i \geqslant \tau$. 
Then there exists a minimal integer $k \geqslant \tau$ such that $\beta_k \neq 0$. 
Define the set
$P_{\infty} = \{ p_0, p_1, \ldots \mid \beta_{p_i} \neq 0\}$.
We now proceed to prove the following identity by induction on $a$:
\begin{equation} \label{main_lemma_req_seq_a(p+d)+p}
\beta_{a(p + d) + p} = \frac{\beta_p}{a + 1},
\ a \geqslant 0, \ p \in P_{\infty}.
\end{equation}
When $a = 0$, the formula is true.
Assume that the formula~\eqref{main_lemma_req_seq_a(p+d)+p} holds for all integers $s < a$. 
We now prove~~\eqref{main_lemma_req_seq_a(p+d)+p} for $a$ applying~\eqref{main_lemma_req_seq_eq},
with $s = p$ and $t = (a - 1)(p + d) + p$:
$$
\beta_{a (p + d) + p} = 
\frac
{\beta_p \beta_{(a - 1)(p + d) + p}}
{\beta_p + \beta_{(a - 1)(p + d) + p}} 
\stackrel{\eqref{main_lemma_req_seq_a(p+d)+p}}{=}
\frac
{\beta_p   \frac{\beta_p}{a}}
{\beta_p + \frac{\beta_p}{a}} =
\frac{\beta_p }{a + 1}.
$$

For an arbitrary $p \in P_{\infty}$ and any $0 < l \leqslant p - \tau$, we have 
$$
\beta_{p - l}\beta_{p + l} =
(\beta_{p - l} + \beta_{p + l})\beta_{1(p + d) + p} =
(\beta_{p - l} + \beta_{p + l})\frac{\beta_p}{2}.
$$
Hence, we conclude that the following equivalence holds:
\begin{equation}\label{main_lemma_req_seq_p-l=p+l}
\beta_{p - l} = 0 \Longleftrightarrow \beta_{p + l} = 0, \ 0<l\leqslant p - \tau.
\end{equation}

For all $p_i \in P_{\infty}$, define $\Delta_i = p_{i + 1} - p_i$, where $i \geqslant 0$.
We will show that $\Delta_i = \Delta_j$ for all $i, j \geqslant 0$.
Let $i \geqslant 0$. Since $p_{i+1} - \Delta_i\in P_\infty$, 
by~\eqref{main_lemma_req_seq_p-l=p+l}, we have $p_{i+1} + \Delta_i\in P_\infty$.
From~\eqref{main_lemma_req_seq_p-l=p+l}, it follows that 
$\beta_{p_{i + 1} + s} = 0$ for all $0 < s < \Delta_i$.
Thus, we conclude that $p_{i + 2} = p_{i + 1} + \Delta_i = p_{i + 1} + \Delta_{i + 1}$,
which implies that $\Delta_i = \Delta_{i + 1}$ for any $i \geqslant 0$.

Based on the previous discussion, we adopt the notation $\Delta = \Delta_0$.
Furthermore, for $p \in P_\infty$, we have
$\beta_p = \beta_{k + \Delta s}$, $s \geqslant 0$.
Thus, $\Delta \mid k + d$, since by~\eqref{main_lemma_req_seq_a(p+d)+p}
we obtain $\beta_{k + d + k} \neq 0$ and $k + d + k = k + \Delta s$ for some $s \geqslant 1$.
By the minimality of $k$, we conclude that $\beta_{k - s} = 0$ for $1 \leqslant s \leqslant k - \tau$. 
Then we obtain $\beta_{k + s} = 0$ by~\eqref{main_lemma_req_seq_p-l=p+l}. 
Therefore, $\Delta > k - \tau$.

Now we prove that
$k - \tau < \Delta \leqslant k + d$ and $\Delta \mid k + d$.
To prove the formula~\eqref{main_lemma_req_seq_solution}, it remains
to show that for any $p\in P_{\infty}$ the equality
$\beta_p = \beta_{k + \Delta s} = \frac{k + d}{k + d + \Delta s}\beta_k$
holds for some $s \geqslant 0$.
This follows from~\eqref{main_lemma_req_seq_a(p+d)+p}:
$$
\frac{\beta_p}{k + d} =
\beta_{(k + d - 1)(p + d) + p} =
\beta_{(p + d)(k + d) - d} = 
\beta_{(p + d - 1)(k + d) + k} =
\frac{\beta_k}{p + d}. \qedhere
$$
\end{proof}

\begin{lemma}[Lemma~\ref{main_theorem_req_seq}]
Let $N \geqslant 0$, $\{d_s\}_{s = N}^{\infty} \subseteq \mathbb{N}$, 
$\{\tau_s\}_{s = N}^{\infty} \subseteq \mathbb{N}$ be fixed parameters,
and suppose that there exists a sequence 
$\{\beta_i^s \mid i \geqslant \tau_s\}_{s = N}^{\infty} \subseteq F$,
satisfying the recurrence relation
\begin{equation} \label{main_theorem_req_seq_eq}
\beta_m^n \beta_t^s = 
\beta_m^n \beta_{m + t + d_n}^s + \beta_t^s \beta_{m + t + d_s}^n, \
m \geqslant \tau_n, \ t \geqslant \tau_s,\ n, s \geqslant N.
\end{equation}
Then one of the following holds:
either $\beta_t^s = 0$ for all $t \geqslant \tau_s$ and $s \geqslant N$
or there exist set $I \subseteq \mathbb{N}$ and integers $\Delta$, $\{ k_i\}_{i \in I}$, 
such that $\beta_{k_i}^i \neq 0$, with 
$0 \leqslant k_i - \tau_i < \Delta \leqslant k_i + d_i$, 
$\Delta \mid k_i + d_i$, $i \in I$,
such that for all $t \geqslant \tau_s$, $s \geqslant N$,
one has
\begin{equation*}
\beta_t^s = 
\begin{cases}
\dfrac{(k_s + d_s)\beta_{k_s}^s}{k_s + d_s + \Delta l}, 
& s \in I, \ t = k_s + \Delta l, \ l \geqslant 0, \\
0, & \mbox{otherwise}.
\end{cases}
\end{equation*}
In particular, if $d_s = \tau_s = 0$, 
then $\beta_t^s = 0$ for all $t \geqslant \tau_s$.
\end{lemma}

\begin{proof}
Define the set $I \subseteq \mathbb{N}$ as follows:
$I = \{ i \mid \mbox{there is } j \mbox{ such that } \beta_j^i \neq 0\}$.
We obtain by~\eqref{main_theorem_req_seq_eq} with $n = s$ that
$$
\beta_m^s \beta_t^s = \big(\beta_m^s + \beta_t^s\big) \beta_{m + t + d_s}^s, \
m, t \geqslant \tau_s,\ s, t \geqslant N.
$$ 

For an arbitrary $i \in I$, by Lemma~\ref{main_lemma_req_seq},
there exists a minimal index $k_i \geqslant \tau_i$, $\beta_{k_i}^i \neq 0$,
and a natural number $\Delta_i \geqslant 1$ such that
\begin{gather} 
k_i - \tau_i < \Delta_i \leqslant k_i + d_i, \quad 
\Delta_i \mid k_i + d_i, \label{main_theorem_req_seq_Deltas}\\
\beta_t^i = 
\begin{cases}
\dfrac{(k_i + d_i)\beta_{k_i}^i}{k_i + d_i + \Delta_i s}, & t = k_i + \Delta_i s, \ s \geqslant 0, \\
0, & \mbox{otherwise}. \label{main_theorem_req_seq_layer_solution}
\end{cases}
\end{gather}
We now check that~\eqref{main_theorem_req_seq_eq} holds for $\beta_t^i$ expressed via the obtained formulas.

{\sc Case 1}.
Let $i, j \in I$ and suppose that $\beta_{k_i + \Delta_i a}^i \neq 0$, 
$\beta_t^j = 0$, $t \geqslant \tau_j$, $a \geqslant 0$.
Substituting the given coefficients into~\eqref{main_theorem_req_seq_eq}, we obtain
$$
\beta_{k_i + \Delta_i a}^i \beta_t^j = 
0 =
\beta_{k_i + \Delta_i a}^i \beta_{k_i + \Delta_i a + t + d_i}^j + 
\beta_t^j \beta_{k_i + \Delta_i a + t + d_j}^i =
\beta_{k_i + \Delta_i a}^i \beta_{k_i + \Delta_i a + t + d_i}^j.
$$
Thus, we must have $\beta_{k_i + \Delta_i a + t + d_i}^j = 0$.
From~\eqref{main_theorem_req_seq_layer_solution}, we obtain the following inequalities:
\begin{gather*}
t \neq k_j + \Delta_j l, \quad
k_i + \Delta_i a + t + d_i \neq k_j + \Delta_j b, \quad 
l, a, b \geqslant 0.
\end{gather*}
We may rewrite these conditions as follows:
\begin{gather*}
\Delta_j \centernot \mid t - k_j, \quad
\Delta_j \centernot \mid t - k_j + k_i + d_i + \Delta_i a, \quad 
a \geqslant 0.  
\end{gather*}
Now, setting  $a = 0$ and applying Lemma~\ref{lemma_kratnost}, we conclude that $\Delta_j \mid k_i + d_i$ and
$\Delta_j \centernot \mid t - k_j + \Delta_i a$, $a \geqslant 0$.
Finally, using Lemma~\ref{lemma_kratnost} again for $a = 1$, we obtain $\Delta_j \mid \Delta_i$. 

For $\beta_{k_j + \Delta_j a}^j \neq 0$ and $\beta_t^i = 0$,
we analogously conclude that 
$\Delta_i \mid \Delta_j$ for all~$i, j \in I$.
Let $\Delta := \Delta_{i_0}$ for some $i_0 \in I$.
Thus, we have $\Delta_i = \Delta$ for any $i \in I$.

{\sc Case 2}.
Let $i, j \in I$ and $\beta_{k_i + \Delta a}^i \neq 0$, 
$\beta_{k_j + \Delta b}^j \neq 0$, $a, b \geqslant 0$.
Since all $\Delta_i$ are mutually equal, there exist $p_i, p_j$
such that $\Delta p_i = k_i + d_i$, $\Delta p_j = k_j + d_j$.
Substituting the described coefficients into~\eqref{main_theorem_req_seq_eq}
and applying~\eqref{main_theorem_req_seq_layer_solution}, we obtain
\begin{multline*}
\frac{\Delta^2 p_i p_j\beta_{k_i}^i\beta_{k_j}^j}{\Delta^2 (p_i + a) (p_j + b)} =
\frac{(k_i + d_i)\beta_{k_i}^i}{k_i + d_i + \Delta a} \cdot
\frac{(k_j + d_j)\beta_{k_j}^j}{k_j + d_j + \Delta b} =
\beta_{k_i + \Delta a}^i \beta_{k_j + \Delta b}^j \\ = 
\beta_{k_i + \Delta a}^i \beta_{k_i + \Delta a + k_j + \Delta_j b + d_i}^j + 
\beta_t^j \beta_{k_i + \Delta a + k_j + \Delta b + d_j}^i \\ =
\frac{(k_i + d_i)\beta_{k_i}^i}{k_i + d_i + \Delta a} \cdot
\frac{(k_j + d_j)\beta_{k_j}^j}{k_j + d_j + \Delta (a + b + p_i)} +
\frac{(k_i + d_i)\beta_{k_i}^i}{k_i + d_i + \Delta a} \cdot
\frac{(k_i + d_i)\beta_{k_i}^i}{k_i + d_i + \Delta (a + b + p_j)} \\ =
\frac{\Delta^2 p_i p_j (a + b + p_i + p_j) \beta_{k_i}^i\beta_{k_j}^j}
{\Delta^2 (p_i + a) (p_j + b) (a + b + p_i + p_j)}.
 \qedhere
\end{multline*}
\end{proof}

\begin{lemma}[Lemma~\ref{k_i_recurrent_lemma_1}]
Let $p, \Delta > 0$ be fixed,
and suppose that there exist two sequences
$\{\xi_{s, t}\}_{s, t = 0}^{\infty} \subseteq \mathbb{N}$ and
$\{k_i\}_{i = 0}^{\infty} \subseteq \mathbb{N}$,
satisfying the recurrence relation
\begin{equation*}
k_{s + t} = k_s + k_t + p - \Delta \xi_{s, t},
\quad s, t \geqslant 0.
\end{equation*}
Then 
$\Delta \xi_{s, 0} = \Delta\xi_{0, s} = k_0 + p$, $s \geqslant 0$, and
\begin{gather*} 
k_n = nk_1 + (n - 1)p - \Delta \sum_{s = 1}^{n - 1} \xi_{1, s}, \quad n > 0,
\\
\xi_{u, v} = -\sum_{s = 1}^{\min\{u, v\} - 1} \xi_{1, s} + 
\sum_{s = \max\{u, v\}}^{u + v - 1} \xi_{1, s}, \quad u, v > 0.
\end{gather*} 
\end{lemma}

\begin{proof}
Substituting $t = 0$ into~\eqref{k_i_recurrent_sequence},  
we obtain $\Delta \xi_{s, 0} = k_0 + p$, $s \geqslant 0$.
For $s = 0$, we have $\xi_{l, 0} = \xi_{0, l}$ for all $l \geqslant 0$.
We prove the formula~\eqref{k_i_recurrent_lemma_1_k_n_solution} by induction on $n > 0$, 
simultaneously proving~\eqref{k_i_recurrent_lemma_1_deltas_solution}.
For $n = 1$, the formula holds trivially. For $n = 2$,
we obtain the identity~\eqref{k_i_recurrent_sequence} with $s = t = 1$.
Assume the induction hypothesis holds for all $l \leqslant n$.
Substituting $s = n$ and $t = 1$ into~\eqref{k_i_recurrent_sequence}, 
and applying the induction hypothesis for $k_n$, we obtain
$$
k_{n + 1} = 
nk_1 + (n - 1)p - \Delta\sum_{s = 1}^{n - 1} \xi_{1, s} +
k_1 + p - \Delta \xi_{1, n} \\ =
(n + 1)k_1 + np - \Delta \sum_{s = 1}^{n} \xi_{1, s}.
$$
Now, consider~\eqref{k_i_recurrent_sequence} 
with~$s = l$ and $t = n + 1 - l$, where $0 < l < \lfloor n / 2 \rfloor$
and $\lfloor x\rfloor$ denotes the integer part of~$x$.
Using the formula~\eqref{k_i_recurrent_lemma_1_k_n_solution} for $k_{n+1}$
and the induction assumption, we get
\begin{multline*}
(n + 1)k_1 + np - \Delta \sum_{s = 1}^{n} \xi_{1, s} = 
k_{n + 1} =
k_l + k_{n + 1 - l} + p - \Delta \xi_{l, n + 1 - l} \\ =
(n + 1)k_1 + np -
\Delta \left(
\sum_{s = 1}^{l - 1} \xi_{1, s} +
\sum_{s = 1}^{n - l} \xi_{1, s} + \xi_{l, n + 1 - l}
\right).
\end{multline*}
Expressing $\xi_{l, n + 1 - l}$ from the above relation, we obtain~\eqref{k_i_recurrent_lemma_1_deltas_solution}.
It is clear that substituting~\eqref{k_i_recurrent_lemma_1_k_n_solution}
and~\eqref{k_i_recurrent_lemma_1_deltas_solution} 
into~\eqref{k_i_recurrent_sequence} yields the required identity.
\end{proof}

\begin{lemma}[Lemma~\ref{k_i_recurrent_lemma_2}] 
Let $r, p, \Delta > 0$ be fixed, 
and suppose that there exist two sequences 
$\{\xi_{s, t}\}_{s, t = 0}^{\infty} \subseteq \mathbb{N}$ and
$\{k_i\}_{i = 0}^{\infty} \subseteq \mathbb{N}$, satisfying the recurrence relation
\begin{equation*}
k_{s + t + r} = k_s + k_t + p - \Delta \xi_{s, t},
\quad s, t \geqslant 0.
\end{equation*}
Then $\xi_{s, 0} = \xi_{0, s}$, $s \geqslant 0$, and following formulas hold:
\begin{gather*}
rk_{ir + j} =
((i + 1)r + j) k_0 + (ir + j)p - 
\Delta \left(
\tau_j + r\sum_{s = 0}^{i - 1}\xi_{0, sr + j} 
\right), \ 0 \leqslant i, \ 0 \leqslant j < r,
\\
\tau_j = 
j\xi_{0, 0} + j\sum_{s = 1}^{r - 1} (\xi_{0, s + 1} - \xi_{1, s}) -
r\sum_{s = 1}^{j - 1} (\xi_{0, s + 1} - \xi_{1, s}),\ 0 \leqslant j < r;
\end{gather*}
\begin{multline*}
\xi_{ar + b, cr + d} =
\sum_{s = 1}^{b - 1} (\xi_{0, s + 1} - \xi_{1, s}) + 
\sum_{s = 1}^{d - 1} (\xi_{0, s + 1} - \xi_{1, s}) - 
\sum_{s = 1}^{b + d - 1} (\xi_{0, s + 1} - \xi_{1, s}) \\ + 
\sum_{s = 0}^{a + c} \xi_{0, sr + b + d} - 
\sum_{s = 0}^{a - 1} \xi_{0, sr + b} - 
\sum_{s = 0}^{c - 1} \xi_{0, sr + d},\ b + d < r,
\end{multline*}
\begin{multline*}
\xi_{ar + b, cr + d} =
\sum_{s = 1}^{b - 1} (\xi_{0, s + 1} - \xi_{1, s}) + 
\sum_{s = 1}^{d - 1} (\xi_{0, s + 1} - \xi_{1, s}) - 
\sum_{s = 1}^{b + d - r - 1} (\xi_{0, s + 1} - \xi_{1, s}) \\ + 
\sum_{s = 0}^{a + c + 1} \xi_{0, sr + b + d - r} - 
\sum_{s = 0}^{a - 1} \xi_{0, sr + b} - 
\sum_{s = 0}^{c - 1} \xi_{0, sr + d} - 
\xi_{0,0} - \sum\limits_{s=1}^{r - 1} (\xi_{0, s + 1} - \xi_{1, s}),\ b + d \geqslant r,
\end{multline*}
where $0 \leqslant a, c, i$, $0 \leq b, d, j < r$. 
\end{lemma}

\begin{proof}
Substituting $s = n$, $t = 0$, and $s = 0$, $t = n$
into~\eqref{k_i_recurrent_sequence2}, where $n \geqslant 0$,
we obtain the relation $\xi_{n, 0} = \xi_{0, n}$ for all $n \geqslant 0$.
From~\eqref{k_i_recurrent_sequence2}, we derive the following equation:
$$
k_0 + k_j + p - \Delta \xi_{0, j} = 
k_{r + j} = 
k_1 + k_{j - 1} + p - \Delta \xi_{1, j - 1}, \quad 1 < j < r.
$$
Multiplying this equation by $r$, we obtain an expression for $rk_j$:
$$
rk_j = rk_{j - 1} + rk_1 - rk_0 + \Delta r (\xi_{0, j} - \xi_{1, j - 1}), 
\quad 1 < j < r.
$$
Consequentially substituting $j = 2, \ldots, r - 1$ into this relation, we find the formula 
\begin{equation} \label{k_i_recurrent_sequence2_intermediate}
rk_j = jrk_1 - (j - 1)rk_0 + 
\Delta r\sum_{s = 1}^{j - 1} (\xi_{0, s + 1} - \xi_{1, s}), 
\quad 1 < j < r.
\end{equation}
From~\eqref{k_i_recurrent_sequence2}, we have $k_r = 2k_0 + p - \Delta \xi_{0, 0}$ 
and can write the relation for $k_{2r}$ as follows:
$$
k_0 +  k_r + p - \Delta \xi_{0, r} 
 = k_{2r} 
 = k_1 + k_{r - 1} + p - \Delta \xi_{1, r - 1}.
$$
Combining like terms and using the derived formula for $k_r$ 
and~\eqref{k_i_recurrent_sequence2_intermediate} for $k_{r - 1}$, we get
\begin{multline*}
0 =
k_0 - k_1 + k_r - k_{r - 1} - \Delta \xi_{0, r} + \Delta \xi_{1, r - 1} =
k_0 - k_1 + 2k_0 + p - \Delta \xi_{0, 0} \\ - 
(r - 1)k_1 + (r - 2)k_0 - \Delta\sum_{s = 1}^{r - 2}(\xi_{0, s + 1} - \xi_{1, s}) - 
\Delta \xi_{0, r} + \Delta \xi_{1, r - 1}.
\end{multline*}
Simplifying this expression, we find an explicit formula for $rk_1$:
\begin{equation} \label{k_i_recurrent_sequence2_rk_1_formula}
rk_1 = (r + 1)k_0 + p - 
\Delta \left(\xi_{0, 0} + \sum_{s = 1}^{r - 1}(\xi_{0, s + 1} - \xi_{1, s})
\right).
\end{equation}

Put~\eqref{k_i_recurrent_sequence2_rk_1_formula}
into~\eqref{k_i_recurrent_sequence2_intermediate}.
This leads to the following formula, which also holds for $j = 0, 1$:
\begin{equation} \label{k_i_recurrent_sequence2_last_base}
rk_j = (r + j)k_0 + jp - \Delta \tau_j, \quad 0 \leqslant j < r.
\end{equation}
In this equality, the term $\tau_j$ is defined 
by the formula~\eqref{k_i_recurrent_sequence2_tau_j}.
We will prove the generalized formula~\eqref{k_i_recurrent_sequence2_last}
by induction on $i$.
The induction base $i = 0$ holds by~\eqref{k_i_recurrent_sequence2_last}.
Assume that~\eqref{k_i_recurrent_sequence2_last} is true 
for all $l \leqslant i$. Prove that the formula holds for $i + 1$:
\begin{multline*}
rk_{ir + j + r} \stackrel{\eqref{k_i_recurrent_sequence2}}{=} 
rk_0 + rk_{ir + j} + rp - \Delta r\xi_{0, ir + j} \\ \stackrel{\eqref{k_i_recurrent_sequence2_last}}{=}
rk_0 + ((i + 1)r + j) k_0  + (ir + j)p - 
\Delta \left(
\tau_j + r\sum_{s = 0}^{i - 1} \xi_{0, sr + j} 
\right) + 
rp - \Delta r\xi_{0, ir + j} \\ 
 = ((i + 2)r + j) k_0 + ((i + 1)r + j)p 
 - \Delta \left(\tau_j + r\sum_{s = 0}^{i}\xi_{0, sr + j} \right).
\end{multline*}

It remains to determine the expression for $\xi_{i, j}$ for all $i, j \geqslant 0$.
Let $0 \leqslant a, c$ and $0 \leqslant b, d < r$.
Then we can rewrite~\eqref{k_i_recurrent_sequence2} as
$$
\xi_{ar + b, cr + d} = \frac{1}{\Delta r}(rk_{ar + b} + rk_{cr + d} + rp - rk_{(a + c + 1)r + b + d}).
$$

Depending on the remainder when $b + d$ is divided by $r$, we distinguish two cases.
Consider $0 \leqslant b + d < r$:
\begin{multline*}
\xi_{ar + b, cr + d} =
\frac{1}{\Delta r}\left(
rk_{ar + b} + rk_{cr + d} + rp - rk_{(a + c + 1)r + b + d} 
\right) \\ \stackrel{\eqref{k_i_recurrent_sequence2_last}}{=}  
\frac{1}{r}\left(
\tau_{b + d} - \tau_b - \tau_d + 
\sum_{s = 0}^{a + c} r\xi_{0, sr + b + d} - 
\sum_{s = 0}^{a - 1} r\xi_{0, sr + b} - 
\sum_{s = 0}^{c - 1}r \xi_{0, sr + d}
\right) \\ \stackrel{\eqref{k_i_recurrent_sequence2_tau_j}}{=}
\sum_{s = 1}^{b - 1} (\xi_{0, s + 1} - \xi_{1, s}) + 
\sum_{s = 1}^{d - 1} (\xi_{0, s + 1} - \xi_{1, s}) - 
\sum_{s = 1}^{b + d - 1} (\xi_{0, s + 1} - \xi_{1, s}) \\ + 
\sum_{s = 0}^{a + c} \xi_{0, sr + b + d} - 
\sum_{s = 0}^{a - 1} \xi_{0, sr + b} - 
\sum_{s = 0}^{c - 1} \xi_{0, sr + d}.
\end{multline*}
Similarly, we obtain the following formulas in the case $r \leqslant b + d < 2r$:
\begin{multline*}
\xi_{ar + b, cr + d} = 
\frac{1}{\Delta r}\left(
rk_{ar + b} + rk_{cr + d} + rp - rk_{(a + c + 2)r + (b + d - r)}
\right) \\ = 
\sum_{s = 1}^{b - 1} (\xi_{0, s + 1} - \xi_{1, s}) + 
\sum_{s = 1}^{d - 1} (\xi_{0, s + 1} - \xi_{1, s}) - 
\sum_{s = 1}^{b + d - r - 1} (\xi_{0, s + 1} - \xi_{1, s}) \\ + 
\sum_{s = 0}^{a + c + 1} \xi_{0, sr + b + d - r} - 
\sum_{s = 0}^{a - 1} \xi_{0, sr + b} - 
\sum_{s = 0}^{c - 1} \xi_{0, sr + d} - 
\xi_{0,0} - 
\sum\limits_{s=1}^{r - 1} (\xi_{0, s + 1} - \xi_{1, s}).
 \qedhere
\end{multline*}
It is straightforward to verify that we obtain an identity, substituting the derived formulas for
$k_i$ and $\xi_{i, j}$, $i, j \geqslant 0$, into~\eqref{k_i_recurrent_sequence2}.
\end{proof}

\section*{Appendix B}

\begin{corollary}[Corollary~\ref{RBO_Case3_corollary}]
Let $R$ be an RB-operator of weight zero on $F[x,y]$ or $F_0[x,y]$
defined by the rule $R(x^n y^m) = \gamma_{n, m} T(x^n y^m)$, 
where $p_x > 0$, $p_y \geqslant 0$, 
and $T$ is an operator from Theorem~\ref{RBO_Case3_main_theorem}.
Then $R$ has one of the following forms.

In Case a), the following formula holds for all $n + m \geqslant \nu(0)$:
\begin{equation*}
\gamma_{n, m} =
\begin{cases}
\dfrac{r\gamma_{k, c}}{r + s}, & 
n = k + \frac{k + p_x}{c + p_y}\Delta s, \ m = c + \Delta s, \  s \geqslant 0,\\
0, & \mbox{otherwise}.
\end{cases}
\end{equation*}

In Case b), when $p_y = 0$, 
we have $\Delta_x \sigma_0 = \Delta_x r = k_0 + p_x$ and
\begin{equation*} 
\gamma_{s, t} =
\begin{cases}
\dfrac{r\gamma_{k_0, 0}\gamma_{k_1, c}}
{rv\gamma_{k_0, 0} + (u + r - \sum_{s = 0}^{v - 1}\sigma_s)\gamma_{k_1, c}}, 
& s = k_v + \Delta u, \ t = cv, \ u, v\geqslant 0, \\
0. & \mbox{otherwise},
\end{cases}
\end{equation*}
for all $s + t \geqslant \nu(0)$.
In Case b), when $0 < p_y$, one has
\begin{gather*}
\gamma_{s, t} =
\begin{cases}
\dfrac{r_y}
{\frac{r_yu - H(v)}{\gamma_{k_0 + \Delta_x, c_y}} + \frac{H(v) + v - r_y(u - 1)}{\gamma_{k_0, c_y}}}, 
& s = k_v + \Delta_x u, \ t = c_y + \Delta_yv, \ u, v\geqslant 0, \\
0, & \mbox{otherwise},
\end{cases}
\\
H(v:= ar_y + b)
 = r_y\sum_{s = 0}^{a - 1}\sigma_{0, sr_y + b} 
 - r_y\sum_{s = 1}^{b - 1} (\sigma_{0, s + 1} {-} \sigma_{1, s}) 
 + b\sigma_{0, 0} + b\sum_{s = 1}^{r_y - 1} (\sigma_{0, s + 1} {-} \sigma_{1, s}), 
\end{gather*}
for all $0 \leqslant a$ and $0 \leqslant b < r_y$.
\end{corollary}

\begin{proof}
We will consider each of the formulas defining the averaging operator
from Theorem~\ref{RBO_Case3_main_theorem}.
By every operator~$T$ from Theorem~\ref{RBO_Case3_main_theorem}
we construct an RB-operator $R_1$ satisfying the condition $\ker R_1 = \ker T$.
Such operator exists by Lemma 3.3 and Remark 3.3 from~\cite{Khodzitskii2}.
Further, we will consider an operator $R_2$ satisfying the condition $\ker T \subsetneq \ker R_2$.
For each $R_2$ we will prove that there exists an averaging operator $P$
described in Theorem~\ref{RBO_Case3_main_theorem} such that $\ker R_2 = \ker P$.
Thus, we will prove that an arbitrary RB-operator~$R$ constructed by an averaging operator~$T$ defined by one from the formulas~\eqref{RBO_Case3a_main_theorem_solution},~\eqref{RBO_Case3b_main_theorem_solution},~\eqref{RBO_Case3b_main_theorem_solution_p_y=0}
has the form~\eqref{RBO_Case3_3_RBO_solution1},~\eqref{RBO_Case3_3_RBO_solution2},~\eqref{RBO_Case3_3_RBO_solution3} respectively.

Throughout the proof, we assume that $R \neq 0$.
Due to the conditions of the statement we have $\ker T \subseteq \ker R$.
We need to determine action of $R$ on all monomials $x^n y^m \not \in \ker T$, $n + m \geqslant \nu(0)$.
Substitute $a = x^n y^m$ and $b = x^s y^t$ into~\eqref{RBO}:
\begin{equation} \label{RBO_Case3}
\gamma_{n, m}\gamma_{s, t} = (\gamma_{n, m}+ \gamma_{s, t})\gamma_{n + s + p_x, m + t + p_y}, 
\quad n + m, s + t \geqslant \nu(0).
\end{equation}
For $\gamma_{n, m}, \gamma_{s, t} \neq 0$, we have the relation
\begin{equation} \label{RBO_Case3_fractions}
\frac{1}{\gamma_{n + s + p_x, m + t + p_y}} = \frac{1}{\gamma_{n, m}} + \frac{1}{\gamma_{s, t}}, 
\quad n + m, s + t \geqslant \nu(0).
\end{equation}
By induction on $a \geqslant 0$ and~\eqref{RBO_Case3},
we obtain the following formulas for any $\gamma_{s, t} \neq 0$ and~$\gamma_{n, m}$:
\begin{gather} \label{RBO_Case3_(n,m)_a+1_raz_generalization}
\gamma_{n + a(s + p_x), m + a(t + p_y)} =
\frac{\gamma_{s, t}\gamma_{n, m}}{\gamma_{s, t} + a \gamma_{n, m}}, \quad 
\gamma_{(a + 1)s + ap_x, (a + 1)t + ap_y} = \gamma_{s, t} / (a + 1), 
\quad a \geqslant 0.
\end{gather}
It is easy to see that~\eqref{RBO} is equivalent to~\eqref{RBO_Case3}.

{\sc Case 1}.
Let us consider an averaging operator $T$ defined by~\eqref{RBO_Case3a_main_theorem_solution}.
Let $\ker T \subseteq \ker R$. 
Denote $\beta_l = \gamma_{k + \frac{k + p_x}{c + p_y}\Delta l, c + \Delta l}$,
$l \geqslant 0$.
Put
$n = k + \frac{k + p_x}{c + p_y}\Delta u$, $m = c + \Delta u$, 
$s = k + \frac{k + p_x}{c + p_y}\Delta v$, $t = c + \Delta v$, 
$u, v \geqslant 0$, in~\eqref{RBO_Case3}:
\begin{multline*}
\gamma_{k + \frac{k + p_x}{c + p_y}\Delta u, c + \Delta u}
\gamma_{k + \frac{k + p_x}{c + p_y}\Delta v, c + \Delta v} =
\beta_u \beta_v = (\beta_u + \beta_v)\beta_{u + v + r} \\ =
(\gamma_{k_c + \frac{k + p_x}{c + p_y}\Delta u, c + \Delta u} +
\gamma_{k_c + \frac{k + p_x}{c + p_y}\Delta v, c + \Delta v})
\gamma_{2k + p_x + \frac{k + p_x}{c + p_y}\Delta (u + v), 2c + p_y + \Delta (u + v)},
\end{multline*}
where $\Delta r = c + p_y$. Such $r$ exists, since $\Delta \mid c + p_y$.
Now we apply Lemma~\ref{main_lemma_req_seq} for $\tau = 0$, $d = r$. 
Since $R \neq 0$, there exists~$\Lambda$ satisfying the conditions 
$w < \Lambda \leqslant w + r$, $\Lambda \mid w + r$, 
and for any $t \geqslant 0$ the following equality holds:
$$
\beta_t = 
\begin{cases}
\dfrac{(w + r)\beta_w}{w + r + \Lambda s}, & t = w + \Lambda s, \ s \geqslant 0, \\
0, & \mbox{otherwise}.
\end{cases}
$$
Due to $\Lambda \mid w + r$, then there exists $l > 0$ such that $\Lambda l = w + r$.
Thus, we obtain~\eqref{RBO_Case3_3_RBO_solution1} 
with parameters $\tilde k = k + \frac{k + p_x}{c + p_y}\Delta w$, 
$\tilde c = c + \Delta w$, $\tilde r = l$, and $\tilde \Delta = \Lambda\Delta$.

{\sc Case 2}.
Before we consider an operator $T$ defined by one
of the formulas~\eqref{RBO_Case3b_main_theorem_solution}
or~\eqref{RBO_Case3b_main_theorem_solution_p_y=0},
we make general remarks concerned Theorem~\ref{RBO_Case3_main_theorem}b.

Denote $\beta_{u, v} = \gamma_{k_v + \Delta_x u, c_y + \Delta_y v}$,
$u, v \geqslant 0$, it means that we consider exactly all nonzero~$\beta_{u, v}$.
Thus, by~\eqref{RBO_Case3_3_averaging_k_i+j+q} and~\eqref{RBO_Case3_fractions} we obtain 
\begin{equation} \label{RBO_Case3_fractions_betas}
\frac{1}{\beta_{u + i + \xi_{v, j}, v+j+r_y}} = 
\frac{1}{\beta_{u, v}} + \frac{1}{\beta_{i, j}}, 
\quad u,v,i,j \geqslant 0.
\end{equation}
Substituting $(u, v, i, j) = (n, m, 0, 0)$ and $(u, v, i, j) = (n, 0, 0, m)$
in~\eqref{RBO_Case3_fractions_betas}, we deduce 
\begin{equation} \label{RBO_Case3_fractions_betas_pretty}
\frac{1}{\beta_{n, m}} = \frac{1}{\beta_{n, 0}} + \frac{1}{\beta_{0, m}} - \frac{1}{\beta_{0, 0}},
\quad n, m \geqslant 0.
\end{equation}

{\sc Case 2.1}a.
Let us consider an averaging operator $T$ defined by~\eqref{RBO_Case3b_main_theorem_solution_p_y=0}.
Let $\ker R = \ker T$.
Recall that the following set of parameters was used for $T$:
$c_y = p_y = 0$, $\Delta = \Delta_x$, $\Delta_y = c$,
$\Delta_x r = k_0 + p_x$, $\xi_{0, s} = r$, $s \geqslant 0$.
Moreover, $\xi_{1, s} = \sigma_s$, $s \geqslant 0$,
and for all $k_i$, $i \geqslant 0$, $\xi_{u, v}$, $u, v > 0$, 
formulas~\eqref{k_i_recurrent_sequence} and~\eqref{k_i_recurrent_lemma_1_deltas_solution} hold.
Hence, we have $r_y = 0$ in~\eqref{RBO_Case3_fractions_betas}.

Relation~\eqref{RBO_Case3_fractions_betas} for $v = j = 0$ has the form 
$\beta_{u, 0} \beta_{i, 0} = (\beta_{u, 0} + \beta_{i, 0}) \beta_{u + i + \xi_{0, 0}, 0}$, $u, i \geqslant 0$. 
We can use Lemma~\ref{main_lemma_req_seq} for $\tau = 0$ and $d = \xi_{0, 0} = r > 0$.
In this case all conditions of Example~\ref{main_lemma_req_seq_example} are fulfilled
and we obtain
\begin{equation} \label{Case3_RB_by_Averaging_beta_s0_formula}
\frac{1}{\beta_{s, 0}} = 
\frac{r + s}{r\beta_{0, 0}}, \quad s \geqslant 0.
\end{equation}

By~\eqref{RBO_Case3_fractions_betas} we have the following relations:
$$
\frac{1}{\beta_{u + r, v}} = \dfrac{1}{\beta_{u, 0}} + \dfrac{1}{\beta_{0, v}}, \quad
\frac{1}{\beta_{r+\xi_{v, w}, v + w}} = \dfrac{1}{\beta_{r, v}} + \dfrac{1}{\beta_{0, w}}, \quad
u, v, w \geqslant 0.
$$
Therefore, we deduce
$$
\frac{1}{\beta_{\xi_{v, w}, 0}} + \frac{1}{\beta_{0, v + w}} =
\frac{1}{\beta_{r+\xi_{v, w}, v + w}} =
\frac{1}{\beta_{r, v}} + \frac{1}{\beta_{0, w}} =
\frac{1}{\beta_{0, 0}} + \frac{1}{\beta_{0, v}} + \frac{1}{\beta_{0, w}}.
$$
Let $v = 1$, 
we express $1 /\beta_{0, w + 1}$ by the formula~\eqref{Case3_RB_by_Averaging_beta_s0_formula} 
applied for $1 / \beta_{\xi_{1, w}, 0}$ as follows:
$$
\frac{1}{\beta_{0, w + 1}} =
\frac{1}{\beta_{0, w}} + \frac{1}{\beta_{0, 1}} - 
\frac{\xi_{1, w}}{r \beta_{0, 0}}, \quad w \geqslant 0.
$$
By induction on $w$ we get the formula
$$
\frac{1}{\beta_{0, v}} = \frac{v}{\beta_{0, 1}} - 
\frac{\sum\limits_{s = 1}^{v - 1}\xi_{1, s}}{r\beta_{0, 0}}, \quad v > 0.
$$
Let us extend the obtained formula to the case $v = 0$,
for this we add and subtract $\xi_{1, 0} = r$ in the numerator of the second fraction:
$$
\frac{1}{\beta_{0, v}} = \frac{v}{\beta_{0, 1}} + 
\frac{r - \sum\limits_{s = 0}^{v - 1}\xi_{1, s}}{r\beta_{0, 0}}, \quad v \geqslant 0.
$$
Apply the obtained expression for $1 / \beta_{0, v}$ and
the formula~\eqref{Case3_RB_by_Averaging_beta_s0_formula} 
in~\eqref{RBO_Case3_fractions_betas_pretty}:
\begin{equation} \label{Case3_RB_by_Averaging_corollary_last_eq_case_p_y=0}
\frac{1}{\beta_{u, v}} = \frac{v}{\beta_{0, 1}} + 
\frac{u + r - \sum\limits_{s = 0}^{v - 1}\xi_{1, s}}{r\beta_{0, 0}}, \quad u, v \geqslant 0.
\end{equation}

Formula~\eqref{Case3_RB_by_Averaging_corollary_last_eq_case_p_y=0} 
describes all nonzero coefficients.
Thus, it is easy to see that to prove that $R$ is an RB-operator,
it remains to check~\eqref{RBO_Case3} only for nonzero coefficients, since $\ker R = \ker T$. 
Note that~\eqref{RBO_Case3_fractions_betas} and~\eqref{RBO_Case3} are equivalent, so we verify~\eqref{RBO_Case3_fractions_betas}.
Denote $A = \frac{1}{\beta_{u, v}} + \frac{1}{\beta_{i, j}}$,
$B = \frac{1}{\beta_{u + i + \xi_{v, j}, v + j}}$ 
and let us show that $A = B$.
By~\eqref{Case3_RB_by_Averaging_corollary_last_eq_case_p_y=0} we have
\begin{gather*}
A = \frac{v {+} j}{\beta_{0, 1}} {+} 
\frac{u + i + 2r - \sum\limits_{s = 0}^{v - 1}\xi_{1, s} - \sum\limits_{s = 0}^{j - 1}\xi_{1, s}}{r\beta_{0, 0}}, \quad
B = \frac{v {+} j}{\beta_{0, 1}} {+} 
\frac{u + i + \xi_{v, j} + r - \sum\limits_{s = 0}^{v + j - 1}\xi_{1, s}}{r\beta_{0, 0}}.
\end{gather*}
Now we use~\eqref{k_i_recurrent_lemma_1_deltas_solution}
for $\xi_{v, j}$ and express $C = r\beta_{0, 0}(A - B)$:
$$
C = \sum_{s = 1}^{\min\{v, j\} - 1} \xi_{1, s} - \sum_{s = \max\{v, j\}}^{v + j - 1} \xi_{1, s} + 
\sum_{s = 0}^{v + j - 1}\xi_{1, s} - \sum_{s = 0}^{v - 1}\xi_{1, s} - 
\sum_{s = 0}^{j - 1}\xi_{1, s} + r.
$$
The equality $\xi_{1, 0} = r$ implies $C = 0$.

{\sc Case 2.1b}.
Now we consider the case of $\ker T \subsetneq \ker R$.
The set of all nonzero coefficients $\alpha_{ij}$ may be described via the following set consisting of pairs of their coordinates:  
$I^2 = \{ (k_s + \Delta l, sc) \mid l, s \geqslant 0\}$, where all $k_s$, $s \geqslant 0$, are defined by~\eqref{k_i_recurrent_sequence}.
To determine $R$, we need to define the set
$J^2 = \{ (i, j) \mid \gamma_{i, j} \neq 0, \ i + j \geqslant \nu(0)\} \subsetneq I^2$.
Let us introduce $P = \{ j \mid \gamma_{i, j} \neq 0$ for some $i \geqslant \nu(j)\} \subseteq \mathbb{N}$.
For all $s \geqslant 0$, $sc \in P$, define a minimal number 
$\eta_s \geqslant k_s$ such that $\gamma_{\eta_s, sc} \neq 0$.

{\sc Case 2.1b.1}. 
Suppose that $\gamma_{l, 0} = 0$ for all $l \geqslant \nu(0)$.
Now we prove that $J^2 = \{ (\eta_s, sc)\}$.
For an arbitrary $\gamma_{s, t} \neq 0$ 
we obtain by~\eqref{RBO_Case3_(n,m)_a+1_raz_generalization} 
\begin{gather} \label{RBO_Case3_2.1_subcase_first_layer_is_zero}
\gamma_{a(s + p_x) + n, at} = 0,\quad
\gamma_{a(s + p_x) + s, (a + 1)t} \neq 0,\
n \geqslant \nu(0),\ a \geqslant 0.
\end{gather}

Suppose that $\gamma_{s, t} \neq 0$ for some $s \geqslant 0$, $t > 0$.
Let us prove by contradiction that 
$\gamma_{s - l, t} = 0$, $0 < l \leqslant s$, 
and $\gamma_{s + l, t} = 0$, $0 < l$.
Assume that $\gamma_{s - d, t} \neq 0$ for some $0 < d \leqslant s$.
Then by~\eqref{RBO_Case3_2.1_subcase_first_layer_is_zero} 
considered for $\gamma_{s - d, t}$, $a = b + 1$,
as well as for $\gamma_{s, t}$, $a = b$, we get
$$
\gamma_{(b + 1)(s - d + p_x) + n, (b + 1)t} = 0, \quad
\gamma_{b(s + p_x) + s, (b + 1)t} \neq 0, 
\ n \geqslant \nu(0), \ b \geqslant 0.
$$
Let us take $b$ so large that $n = d - p_x + bd > 0$ holds.
For such $n$ we get a contradiction:
$$
0 = \gamma_{(b + 1)(s - d + p_x) + d - p_x + bd, (b + 1)t} 
  = \gamma_{b(s + p_x) + s, (b + 1)t} \neq 0.
$$
Assume that $\gamma_{s + d, t} \neq 0$ for some $d > 0$.
Analogously to the previous case we obtain by~\eqref{RBO_Case3_2.1_subcase_first_layer_is_zero}
$$
\gamma_{(b + 1)(s + p_x) + n, (b + 1)t} = 0, \quad
\gamma_{b(s + d + p_x) + s + d, (b + 1)t} \neq 0, 
\ n \geqslant \nu(0), \ b \geqslant 0,
$$
and we take $n = d - p_x + bd> 0$ for a suitable~$b$.
As in the previous case, we arrive at a~contradiction.

Further proof repeats the steps from the proof of Case~1 of Theorem~\ref{RBO_Case3_main_theorem}.
Thus, all nonzero coefficients $\gamma_{n, m} \neq 0$ of the operator $R$
are defined by the formulas $n = k + \frac{k + p_x}{w + p_y}\Lambda s$,
$m = w + \Lambda s$, $s \geqslant 0$, for some $k, w, \Lambda \geqslant 0$ .
In this case there is an averaging operator $T_1$ defined by~\eqref{RBO_Case3a_main_theorem_solution}
such that $\ker T_1 = \ker R$.
So, we reduce the problem to the already considered Case~1.

{\sc Case 2.1b.2}.
Suppose that $\gamma_{\eta_0, 0} \neq 0$ for some $\eta_0$.
Then by Lemma~\ref{main_lemma_req_seq} there exists some $\Lambda > 0$ 
such that for an arbitrary~$s$ satisfying the condition $\gamma_{s, 0} \neq 0$,
the formula $s = \eta_0 + \Lambda l$ holds for some $l \geqslant 0$.
Also note that $\Delta \mid \Lambda$.
Now we want to show that
\begin{equation} \label{RBO_Case3_corollary_subcase_2.1_2_J^2}
J^2 = \{(s, t)\mid s = \eta_t + \Lambda n,\, t = wm,\, n, m \geqslant 0\},
\end{equation}
where $w > 0$ is the minimal number such that 
there exists $\gamma_{l, w} \neq 0$ for some $l \geqslant 0$.
In particular, $c \mid w$.

First of all we prove that $w \mid m$ for any $\gamma_{n, m} \neq 0$.
Minimality of~$w$ implies that $\gamma_{n, m} = 0$ for all $0 \leqslant n$ and $0 < m < w$.
Thus, considering~\eqref{RBO_Case3_(n,m)_a+1_raz_generalization} for $s = \eta_w$ and $t = w$, we get
\begin{equation} \label{RBO_Case3_corollary_(n,m)_a+1_raz_generalization_modified}
\gamma_{a(\eta_w + p_x) + n, m + aw} = 0, \quad
\gamma_{(a + 1)\eta_w + ap_x, (a + 1)w} \neq 0,\
0 < m < w,\ 0 \leqslant n, a.
\end{equation}
Let us show that $\gamma_{n, aw + m} = 0$ for all $0 < m < w$ and $0 \leqslant n, a$.
Assume to the contrary that this is not true.
Let $dw + m$ be a minimal number such that $0 < m < w$, $0 \leqslant d$, 
and $\gamma_{n, dw + m} \neq 0$ for some $n \geqslant 0$.
Considering~\eqref{RBO_Case3} for $0 < w - m < w$ and $0 \leqslant s$, we obtain
$$
0 =
\gamma_{n, dw + m} \gamma_{s, w - m} = 
(\gamma_{n, dw + m} + \gamma_{s, w - m})\gamma_{n + s + p_x, dw + w} =
\gamma_{n, dw + m} \gamma_{n + s + p_x, dw + w},
$$
which implies $\gamma_{n + s + p_x, (d + 1)w} = 0$.
Now we apply~\eqref{RBO_Case3} for 
$n = a(\eta_0 + p_x) + \eta_0$, $m = 0$, $s = (d + 1)\eta_w + dp_x$, $t = (d + 1)w$:
$$
\gamma_{a(\eta_0 + p_x) + \eta_0 + (d + 1)\eta_w + dp_x + p_x, (d + 1)w} =
\frac{\gamma_{a(\eta_0 + p_x) + \eta_0, 0} \gamma_{(d + 1)\eta_w + dp_x, (d + 1)w}}
{\gamma_{a(\eta_0 + p_x) + \eta_0, 0} + \gamma_{(d + 1)\eta_w + dp_x, (d + 1)w}} \neq 0, \
a \geqslant 0.
$$
Here $\gamma_{a(\eta_0 + p_x) + \eta_0, 0}, \gamma_{(d + 1)\eta_w + dp_x, (d + 1)w} \neq 0$
by~\eqref{RBO_Case3_corollary_(n,m)_a+1_raz_generalization_modified}.
For given $n$ and $d$ one may find $s, a \geqslant 0$ 
such that $n + s + p_x = a(\eta_0 + p_x) + \eta_0 + (d + 1)\eta_w + dp_x + p_x$.
Hence, we get a~contradiction: $0 = \gamma_{n + s + p_x, (d + 1)w} \neq 0$.
So, $w \mid t$ for any $t$ such that $\gamma_{s, t}\neq0$.

In order to complete the proof of the formula~\eqref{RBO_Case3_corollary_subcase_2.1_2_J^2}
it remains to show that there exists~$\Lambda$ such that if $\gamma_{s, t} \neq 0$, 
then $s = \eta_t + \Lambda l$ for some $l \geqslant 0$.
Suppose that $d$ is the minimal number 
such that $\gamma_{\eta_{dw} + n, dw} \neq 0$ and $\Lambda \centernot \mid n$.
The latter implies $\gamma_{\eta_0 + n, 0} = 0$,
and so by~\eqref{RBO_Case3} we get
\begin{gather*}
0 = \gamma_{\eta_{dw}, dw}\gamma_{\eta_0 + n, 0} = 
\gamma_{\eta_{dw}, dw}\gamma_{\eta_{dw} + \eta_0 + n + p_x, dw},\\
0 \neq \gamma_{\eta_{dw} + n, dw}\gamma_{\eta_0, 0} = 
(\gamma_{\eta_{dw} + n, dw} + \gamma_{\eta_0, 0})\gamma_{\eta_{dw} + \eta_0 + n + p_x, dw}.
\end{gather*}
Therefore, we arrive at~a contradiction: $0 = \gamma_{\eta_{dw} + \eta_0 + n + p_x, dw} \neq 0$.

It is easy to see, as in the proof of Theorem~\ref{RBO_Case3_main_theorem} (Step~7),
we have reduced the whole question to the problem of finding the parameters $\eta_s$.
We apply Lemma~\ref{k_i_recurrent_lemma_1} to express~$\eta_s$.
Thus, there exists an averaging operator $T_2$
defined by~\eqref{RBO_Case3b_main_theorem_solution_p_y=0}
such that $\ker T_2 = \ker R$.
All that needs to be done is to repeat similar arguments from Case~2.1a.

{\sc Case 2.2a}.
Let us consider an averaging operator $T$ defined by~\eqref{RBO_Case3b_main_theorem_solution}.
Let $\ker R = \ker T$.
Using~\eqref{RBO_Case3_fractions_betas} twice, 
where instead of $(u, v)$, $(i, j)$ we put 
$(0, v)$, $(u, v)$ and $(1, v)$, $(u - 1, v)$ respectively, 
we obtain
$$
\frac{1}{\beta_{u, v}} = \frac{1}{\beta_{1, v}} + \frac{1}{\beta_{u - 1, v}} - \frac{1}{\beta_{0, v}},
\quad v \geqslant 0, \quad u \geqslant 2.
$$
The induction on $u \geqslant 2$ allows us to derive from the last equality the next one:
\begin{equation} \label{Case3_RB_by_Averaging_corollary_beta_uv}
\frac{1}{\beta_{u, v}} = \frac{u}{\beta_{1, v}} - \frac{u - 1}{\beta_{0, v}}.
\end{equation}
Note that the formula~\eqref{Case3_RB_by_Averaging_corollary_beta_uv}
holds for $u = 0, 1$ too.
So, we have $u \geqslant 0$ in~\eqref{Case3_RB_by_Averaging_corollary_beta_uv}.

By~\eqref{k_i_recurrent_sequence2}, for any $v \geqslant 2$ we have
$$
k_{r_y + v} = k_0 + k_v + p_x - \Delta_x \xi_{0, v} = k_1 + k_{v - 1} + p_x - \Delta_x \xi_{1, v - 1},
$$
i.\,e. $k_0 + k_v + \Delta_x \xi_{1, v -1} = k_1 + k_{v - 1} + \Delta_x \xi_{0, v}$.
Recall that due to the definition, we have
\begin{gather*}
\beta_{\xi_{1, v - 1}, 0} = \gamma_{k_0 + \xi_{1, v - 1}, c_y}, \quad
\beta_{0, v} = \gamma_{k_v, c_y + \Delta_y v}, \\
\beta_{\xi_{0, v}, 1} = \gamma_{k_1 + \Delta_x \xi_{0, v}, c_y + \Delta_y}, \quad
\beta_{0, v - 1} = \gamma_{k_{v - 1}, c_y + \Delta_y (v - 1)}.
\end{gather*}
For these coefficients, we get by~\eqref{RBO_Case3_fractions_betas} the equality
$$
\frac{1}{\beta_{\xi_{1, v - 1}, 0}} + \frac{1}{\beta_{0, v}} =
\frac{1}{\beta_{\xi_{0, v}, 1}} + \frac{1}{\beta_{0, v - 1}},
\quad v \geqslant 2.
$$
Now, applying~\eqref{Case3_RB_by_Averaging_corollary_beta_uv} 
for $1 / \beta_{\xi_{0, v}, 1}$ and $1 / \beta_{\xi_{1, v - 1}, 0}$
and extracting $1 /\beta_{0, v}$, we arrive at
$$
\frac{1}{\beta_{0, v}} = 
\dfrac{\xi_{1, v - 1} - 1}{\beta_{0, 0}} -
\dfrac{\xi_{0, v} - 1}{\beta_{0, 1}} -
\dfrac{\xi_{1, v - 1}}{\beta_{1, 0}} + 
\dfrac{\xi_{0, v} }{\beta_{1, 1}} +
\dfrac{1}{\beta_{0, v - 1}}, \quad v \geqslant 2.
$$
By induction on $v$ we get
$$
\frac{1}{\beta_{0, v}} 
 = \frac{\sum\limits_{s = 1}^{v - 1} \xi_{1, s} - (v - 1)}{\beta_{0, 0}} 
 - \frac{\sum\limits_{s = 2}^{v} \xi_{0, s} - v}{\beta_{0, 1}} 
 - \frac{\sum\limits_{s = 1}^{v - 1} \xi_{1, s}}{\beta_{1, 0}} 
 + \frac{\sum\limits_{s = 2}^{v} \xi_{0, s}}{\beta_{1, 1}}, \quad v \geqslant 2.
$$
Let us combine like terms and note that the obtained formula is also true for $v = 0, 1$:
$$
\frac{1}{\beta_{0, v}} =
\sum_{s = 1}^{v - 1} \xi_{1, s} \left(\frac{1}{\beta_{0, 0}} - \frac{1}{\beta_{1, 0}}\right) +
\sum_{s = 2}^{v} \xi_{0, s} \left(\frac{1}{\beta_{1, 1}} - \frac{1}{\beta_{0, 1}}\right) +
\frac{v}{\beta_{0, 1}} - \frac{v - 1}{\beta_{0, 0}}, \quad v \geqslant 0.
$$
Now we substitute the derived formula for $1 / \beta_{0, v}$
and~\eqref{Case3_RB_by_Averaging_corollary_beta_uv} for $1/\beta_{u, 0}$
in~\eqref{RBO_Case3_fractions_betas_pretty}:
\begin{multline*}
\frac{1}{\beta_{u, v}} =
\sum_{s = 1}^{v - 1} \xi_{1, s} \left(\frac{1}{\beta_{0, 0}} - \frac{1}{\beta_{1, 0}}\right) +
\sum_{s = 2}^{v} \xi_{0, s} \left(\frac{1}{\beta_{1, 1}} - \frac{1}{\beta_{0, 1}}\right) + \frac{v}{\beta_{0, 1}} + \frac{u}{\beta_{1, 0}} - \frac{u + v - 1}{\beta_{0, 0}}, \ u, v \geqslant 0.
\end{multline*}
Putting $u = v = 1$ in the last formula, we obtain 
$1/\beta_{1, 1} = 1/\beta_{0, 1} + 1/\beta_{1, 0} - 1/\beta_{0, 0}$.
Taking into account this relation, we finally write down the formula
\begin{multline} \label{Case3_RB_by_Averaging_corollary_hardest_case_last_formula}
\frac{1}{\beta_{u, v}} =
\left(\sum_{s = 1}^{v - 1}\xi_{1, s} - \sum_{s = 2}^{v} \xi_{0, s}\right)
\left( \frac{1}{\beta_{0, 0}} - \frac{1}{\beta_{1, 0}}\right) +
\frac{v}{\beta_{0, 1}} + \frac{u}{\beta_{1, 0}} - \frac{u + v - 1}{\beta_{0, 0}}, \ u, v \geqslant 0.
\end{multline}

Our current goal is  to express all $\xi_{1, s}$, $s \geqslant 0$,
via $\xi_{1, l}$, $0 \leqslant l \leqslant r_y - 1$, and $\xi_{0, l}$, $l \geqslant 0$.
For this, we apply formulas~\eqref{k_i_recurrent_sequence2_deltas1}
and~\eqref{k_i_recurrent_sequence2_deltas2}:
\begin{gather} \label{formula_for_deltas_1,s_with_r_y>1}
\begin{gathered}
\xi_{1, ur_y + v} = 
\xi_{1, v} + \xi_{0, ur_y + v + 1} -  \xi_{0, v + 1} +
\sum_{s = 0}^{u - 1} (\xi_{0, sr_y + v + 1} -  \xi_{0, sr_y + v}),
\quad 0 \leqslant v < r_y - 1, \\
\xi_{1, (u + 1)r_y - 1} = 
\xi_{1, r_y - 1} + 
\xi_{0, ur_y} +
\xi_{0, (u + 1)r_y} - 
\xi_{0, r_y} - 
\xi_{0, 0} +
\sum_{s = 0}^{u - 1}(\xi_{0, sr_y} - \xi_{0, (s + 1)r_y - 1}).
\end{gathered}
\end{gather}

We want to express $S(v - 1) = \sum\limits_{s = 1}^{v - 1} \xi_{1, s}$ 
in~\eqref{Case3_RB_by_Averaging_corollary_hardest_case_last_formula} 
by~\eqref{formula_for_deltas_1,s_with_r_y>1}.
Note that $S(v - 1) = 0$ for $v = 0, 1$.
Consider $(2\leqslant)v = ur_y + w$, where $0 \leqslant u$ and $0 \leqslant w < r_y$.
So
\begin{equation} \label{S(v)_case3_formula}
S(v - 1) =
\sum_{k = 0}^{u - 1}\sum_{d = 0}^{r_y - 1} \xi_{1, kr_y + d} 
+ \sum_{d = 0}^{w - 1} \xi_{1, ur_y + d} - \xi_{1, 0}, \quad v \geqslant 2.
\end{equation}
By~\eqref{formula_for_deltas_1,s_with_r_y>1} considered for 
$\xi_{1, kr_y + d}$, $0 \leqslant k$, $0 \leqslant d < r_y$,
we have
\begin{multline*} 
\sum_{d = 0}^{r_y - 1}\xi_{1, kr_y + d} =
\sum_{d = 0}^{r_y - 2}\xi_{1, kr_y + d} + \xi_{1, (k + 1)r_y - 1} 
\\\stackrel{\eqref{formula_for_deltas_1,s_with_r_y>1}}{=}
\sum_{d = 0}^{r_y - 2}
\left(
\xi_{1, d} + \xi_{0, kr_y + d + 1} - \xi_{0, d + 1} +
\sum_{s = 0}^{k - 1}(\xi_{0, sr_y + d + 1} - \xi_{0, sr_y + d})
\right) + 
\xi_{1, (k + 1)r_y - 1} 
\\\stackrel{\eqref{formula_for_deltas_1,s_with_r_y>1}}{=}
\sum_{d = 0}^{r_y - 2}(\xi_{1, d} + \xi_{0, kr_y + d + 1} - \xi_{0, d + 1}) +
\sum_{s = 0}^{k - 1}(\xi_{0, (s + 1)r_y - 1} - \xi_{0, sr_y})
\\ + \xi_{1, r_y - 1} + \xi_{0, kr_y} + \xi_{0, (k + 1)r_y} 
- \xi_{0, r_y} - \xi_{0, 0} +
\sum_{s = 0}^{k - 1} ( \xi_{0, sr_y} - \xi_{0, (s + 1)r_y - 1})
\\ =
\sum_{d = 0}^{r_y - 1}(\xi_{1, d} + \xi_{0, kr_y + d + 1} - \xi_{0, d + 1}) + 
\xi_{0, kr_y} - \xi_{0, 0}, \quad k \geqslant 0.
\end{multline*}
Now we add up obtained expressions from $0$ to $u - 1$:
\begin{multline} \label{S(v)_case3_formula_sums}
\sum_{k = 0}^{u - 1}\sum_{d = 0}^{r_y - 1}\xi_{1, kr_y + d} =
\sum_{k = 0}^{u - 1}\sum_{d = 0}^{r_y - 1}(\xi_{1, d} + \xi_{0, kr_y + d + 1} - \xi_{0, d + 1}) +
\sum_{k = 0}^{u - 1}(\xi_{0, kr_y} - \xi_{0, 0}) \\ =
u \sum_{d = 0}^{r_y - 1}(\xi_{1, d} - \xi_{0, d + 1}) +
\sum_{s = 1}^{ur_y} \xi_{0, s} + 
\sum_{k = 0}^{u - 1}\xi_{0, kr_y} - u\xi_{0, 0}.
\end{multline}
Hence we have rewritten the first term from the right-hand side of~\eqref{S(v)_case3_formula}.
Now we express the second term:
$$
\sum_{d = 0}^{w - 1}\xi_{1, ur_y + d} \stackrel{\eqref{formula_for_deltas_1,s_with_r_y>1}}{=}
\sum_{d = 0}^{w - 1}(\xi_{1, d} - \xi_{0, d + 1}) +
\sum_{s = ur_y + 1}^{ur_y + w} \xi_{0, s} +
\sum_{s = 0}^{u - 1}(\xi_{0, sr_y + w} - \xi_{0, sr_y}).
$$
Adding this expression to~\eqref{S(v)_case3_formula_sums} and subtracting $\xi_{1, 0}$, we finally obtain
$$
S(v - 1) = 
u\sum_{d = 0}^{r_y - 1}(\xi_{1, d} - \xi_{0, d + 1}) +
\sum_{d = 0}^{w - 1}(\xi_{1, d} - \xi_{0, d + 1}) +
\sum_{s = 1}^{ur_y + w} \xi_{0, s} + 
\sum_{s = 0}^{u - 1}\xi_{0, sr_y + w} - u\xi_{0, 0} - \xi_{1, 0}.
$$ 
Introduce $F(v) = S(v - 1) - \sum\limits_{s = 2}^v \xi_{0, s}$.
Applying the equality $\xi_{1, 0} = \xi_{0, 1}$ fulfilled by Lemma~\ref{k_i_recurrent_lemma_2},
we get for all $0 \leqslant n$ and $0 \leqslant m < r_y$ that
\begin{equation} \label{RBO_Case3_F(v)_formula}
F(nr_y + m) =
\sum_{s = 0}^{n - 1}\xi_{0, sr_y + m} -
\sum_{s = 0}^{m - 1}(\xi_{0, s + 1} - \xi_{1, s}) -
n\sum_{s = 0}^{r_y - 1}(\xi_{0, s + 1} - \xi_{1, s}) - n\xi_{0, 0}.
\end{equation}

By the definition, 
$F(v) = 0$ for $v=0,1$, and it is consistent with the formula~\eqref{RBO_Case3_F(v)_formula}, since
$\xi_{1, 0} = \xi_{0, 1}$.
Hence, we can rewrite~\eqref{Case3_RB_by_Averaging_corollary_hardest_case_last_formula} as
\begin{equation} \label{Case3_RB_by_Averaging_corollary_hardest_case_last_formula_with_F(v)}
\frac{1}{\beta_{u, v}} =
\frac{F(v) - u - v + 1}{\beta_{0, 0}} - \frac{F(v) - u}{\beta_{1, 0}} + \frac{v}{\beta_{0, 1}},
\quad u, v \geqslant 0.
\end{equation}
Now we express all summands from~\eqref{RBO_Case3_fractions_betas} by~\eqref{Case3_RB_by_Averaging_corollary_hardest_case_last_formula_with_F(v)}:
\begin{multline*}
\frac{F(v) - u - v + 1 + F(j) - i - j + 1}{\beta_{0, 0}} - 
\frac{F(v) - u + F(j) - i}{\beta_{1, 0}} + 
\frac{v + j}{\beta_{0, 1}} \\ -
\frac{F(v + j + r_y) - u - i - \xi_{v, j} - v - j - r_y + 1}{\beta_{0, 0}} \\ + 
\frac{F(v + j + r_y) - u - i - \xi_{v, j}}{\beta_{1, 0}} - 
\frac{v + j + r_y}{\beta_{0, 1}} = 0.
\end{multline*}
So we get the relation
\begin{equation} \label{RBO_Case3_base_betas_relation}
\frac{r_y}{\beta_{0, 1}} =
\left( F(v + j + r_y) - F(v) - F(j) - \xi_{v, j}\right)
\left( \frac{1}{\beta_{1, 0}} - \frac{1}{\beta_{0, 0}}\right) +
\frac{r_y + 1}{\beta_{0, 0}}.
\end{equation}
Now we find the value of $F(v + j + r_y) - F(v) - F(j)$.
For this, we apply~\eqref{RBO_Case3_F(v)_formula}
for $v = ar_y + b$, $j = cr_y + d$, where $0 \leqslant a, c$, $0 \leqslant b, d < r_y$.
We have two cases: $b + d < r_y$ and $b + d \geqslant r_y$.

Suppose that $b + d < r_y$, then
\begin{multline*}
F(v + j + r_y) - F(v) - F(j) = 
F((a + c + 1)r_y + b + d) - F(ar_y + b) - F(cr_y + d) \allowdisplaybreaks
\\\stackrel{\eqref{RBO_Case3_F(v)_formula}}{=}
\sum_{s = 0}^{b - 1}(\xi_{0, s + 1} - \xi_{1, s}) +
\sum_{s = 0}^{d - 1}(\xi_{0, s + 1} - \xi_{1, s}) -
\sum_{s = 0}^{b + d - 1}(\xi_{0, s + 1} - \xi_{1, s}) \\ +
\sum_{s = 0}^{a + c}\xi_{0, sr_y + b + d} -
\sum_{s = 0}^{a - 1}\xi_{0, sr_y + b} -
\sum_{s = 0}^{c - 1}\xi_{0, sr_y + d} -
\sum_{s = 0}^{r_y - 1}(\xi_{0, s + 1} - \xi_{1, s}) -
\xi_{0, 0} \\
\stackrel{\eqref{k_i_recurrent_sequence2_deltas1}}{=}
\xi_{v, j} - \sum_{s = 0}^{r_y - 1}(\xi_{0, s + 1} - \xi_{1, s}) - \xi_{0, 0}.
\end{multline*}
The same formula holds in the case $b + d \geqslant r_y$,
here we apply~\eqref{k_i_recurrent_sequence2_deltas2}.
Therefore we may rewrite~\eqref{RBO_Case3_base_betas_relation} as follows:
$$
\frac{r_y}{\beta_{0, 1}} =
\frac{\sum\limits_{s = 0}^{r_y - 1}\xi_{1, s} - \sum\limits_{s = 0}^{r_y}\xi_{0, s}}{\beta_{1, 0}} \\ -
\frac{\sum\limits_{s = 0}^{r_y - 1}\xi_{1, s} - \sum\limits_{s = 0}^{r_y}\xi_{0, s} - r_y - 1}{\beta_{0, 0}}.
$$

Substituting the found expression for $1/\beta_{0, 1}$ in~\eqref{Case3_RB_by_Averaging_corollary_hardest_case_last_formula_with_F(v)},
we get the final formula
\begin{gather*}
\frac{1}{\beta_{u, v}} =
\frac{H(v) + v -r_y(u - 1)}{r_y\beta_{0, 0}} +
\frac{r_yu - H(v)}{r_y\beta_{1, 0}},
\quad u, v \geqslant 0,
\end{gather*}
where $H(v)$ is defined by~\eqref{RBO_Case3_3_RBO_solution3_H(v)}
for $\xi_{0, i} = \sigma_{0, i}$, $1 \leqslant i$, 
and $\xi_{1, i} = \sigma_{1, i}$, $1 \leqslant i \leqslant r_y - 1$, and the formula $H(v) = r_yF(v) + v\sum\limits_{s = 0}^{r_y}\xi_{0, s} - v\sum\limits_{s = 0}^{r_y - 1}\xi_{1, s}$ holds.

{\sc Case 2.2b}.
It remains to consider the case of $\ker T \subsetneq \ker R$.
All nonzero coefficients of~$T$ are described by the set
$I^2 = \{ (k_j + \Delta_x i, c_y + \Delta_y j) \mid i, j \geqslant 0\}$
and~\eqref{k_i_recurrent_sequence2} holds for all $k_s$, $s \geqslant 0$.
As in Case~2.1b, let us define the sets
$J^2 = \{ (i, j) \mid \gamma_{i, j} \neq 0, \ i + j \geqslant \nu(0)\} \subsetneq I^2$
and $P = \{ j \mid \gamma_{i, j} \neq 0$ for some $i \geqslant \nu(j)\} \subseteq \mathbb{N}$.
For all $s \geqslant 0$, $c_y + \Delta_y s \in P$, define a minimal number 
$\eta_{c_y + \Delta_y s} \geqslant k_s$ such that $\gamma_{\eta_{c_y + \Delta_y s}, c_y + \Delta_y s} \neq 0$.
Let $w$ be a minimal number from $P$, i.\,e. $\gamma_{\eta_w, w} \neq 0$.

{\sc Case 2.2b.1}.
Suppose that $\gamma_{\eta_w + l, w} = 0$ for all $l > 0$.
As it was done in Case~2.1b.1, we reduce this case 
to the analogous steps from Theorem~\ref{RBO_Case3_main_theorem} 
and Case~1 by showing that $J^2 = \{ (\eta_j, j) \mid j \in P\}$.
To the contrary, suppose that there exist 
$w < d$, $0 \leqslant a < b$ such that 
$\gamma_{a, d}, \gamma_{b, d} \neq 0$. 
By~\eqref{RBO_Case3} we obtain
\begin{gather*}
0 = \gamma_{\eta_w + u, w} \gamma_{a, d} = \gamma_{a, d}\gamma_{\eta_w + u + a + p_x, w + d + p_x}, \quad u > 0\\
0 \neq \gamma_{\eta_w, w} \gamma_{b, d} = (\gamma_{\eta_w, w} + \gamma_{b, d})\gamma_{\eta_w + b + p_x, w + d + p_x}.
\end{gather*}
Thus, we get a contradiction for $u = b - a$.

{\sc Case 2.2b.2}.
In this case there exists a~minimal $\Lambda_x > 0$ 
such that $\gamma_{\eta_w + \Lambda_x, w} \neq 0$; note that $\Delta_x \mid \Lambda_x$.
The formula~\eqref{RBO_Case3}
considered for $\gamma_{\eta_w, w}^2$ implies that $2w + p_x \in P$.
Hence, there exists a minimal $\Lambda_y > 0$ 
such that $w + \Lambda_y \in P$ and $\Delta_y \mid \Lambda_y$.
Without loss of generality, we can work with the parameter $\Lambda_y$
like with the parameter $\Delta_y$ in Theorem~\ref{RBO_Case3_main_theorem}
and obtain analogous conclusions about $\Lambda_y$ as about $\Lambda_x$.

Let us consider an arbitrary $s \in P$.
Show that $\gamma_{\eta_s + \Lambda_x, s} \neq 0$,
it follows from~\eqref{RBO_Case3}:
$$
\frac{\gamma_{\eta_s + \Lambda_x, s} \gamma_{\eta_w, w}}
{\gamma_{\eta_s + \Lambda_x, s} + \gamma_{\eta_w, w}} =
\gamma_{\eta_s + \eta_w + \Lambda_x + p_x, s + w + p_y} =
\frac{\gamma_{\eta_s, s} \gamma_{\eta_w + \Lambda_x, w}}
{\gamma_{\eta_s, s} + \gamma_{\eta_w + \Lambda_x, w}} \neq 
0.
$$
Now we consider $\gamma_{\eta_s, s}, \gamma_{\eta_s + \Lambda_x, s} \neq 0$.
Let us prove by induction on $a$ that $\gamma_{\eta_s + \Lambda_x a, s} \neq 0$ for all $a \geqslant 0$.
The base for $a = 0, 1$ is trivial.
Now we assume that the induction hypothesis holds for all $b \leqslant a$. 
To prove the induction hypothesis for $a + 1$, 
like as in the previous paragraph, we consider~\eqref{RBO_Case3}:
$$
\frac{\gamma_{\eta_s + \Lambda_x a, s} \gamma_{\eta_s, s}}
{\gamma_{\eta_s + \Lambda_x a, s} + \gamma_{\eta_s, s}} =
\gamma_{2\eta_s + \Lambda_x a + p_x, 2s + p_y} =
\frac{\gamma_{\eta_s + \Lambda_x (a - 1), s} \gamma_{\eta_s + \Lambda_x, s}}
{\gamma_{\eta_s + \Lambda_x (a - 1), s} + \gamma_{\eta_s + \Lambda_x, s}} \neq 
0.
$$

Note that $\gamma_{\eta_s + \Lambda_x a + l, s} = 0$ for all $0 < l < \Lambda_x$
and $0 < a$. Let be $l$ an arbitrary natural with $0 < l < \Lambda_x$,
therefore by~\eqref{RBO_Case3} we have
$$
0 =
\gamma_{\eta_s + \Lambda_x a + l, s}\gamma_{\eta_s + \Lambda_x - l, s} =
(\gamma_{\eta_s + \Lambda_x a + l, s} + \gamma_{\eta_s + \Lambda_x - l, s})
\gamma_{2\eta_s + \Lambda_x (a + 1) + p_x, 2s + p_y},
$$
since $\gamma_{\eta_s + \Lambda_x - l, s} = 0$
and $\gamma_{2\eta_s + \Lambda_x (a + 1) + p_x, 2s + p_y} \neq 0$
we obtain $\gamma_{\eta_s + \Lambda_x a + l, s} = 0$.

Like in the proof of Theorem~\ref{RBO_Case3_main_theorem}, 
it is not hard to understand that analogous arguments hold for the parameter $\Lambda_y > 0$.
Thus, if $s \in P$, then $s + \Lambda_y \in P$.
So, the set~$J^2$ described as follows: 
$J^2 = \{ (\eta_{w + \Lambda_yt} + \Lambda_x s, w + \Lambda_yt)\mid s, t \geqslant 0\}$.
We reduce the problem to the search of the parameters $\eta_i$, $i \geqslant 0$.
Therefore, we need to apply Lemma~\ref{k_i_recurrent_lemma_2}.
Thus, there exists an averaging operator $T_2$
defined by~\eqref{RBO_Case3b_main_theorem_solution}
such that $\ker T_2 = \ker R$ and we can describe $T_2$ due to Case~2.2a.
\end{proof}
\end{document}